\documentclass{amsart}
\usepackage[utf8]{inputenc}
\usepackage{amsmath}
\usepackage{amssymb}
\usepackage[foot]{amsaddr}
\usepackage{todonotes}
\usepackage{mathtools}
\usepackage{stmaryrd}   
\usepackage{enumitem}
\usepackage{xcolor}
\usepackage{graphicx}
\usepackage{subcaption}
\usepackage{mathrsfs}   
\usepackage{nicefrac}
\usepackage{cite}
\usepackage{url}
\usepackage[framemethod=tikz]{mdframed}	

\newtheorem{theorem}{Theorem}

\newtheorem{lemma}[theorem]{Lemma}
\newtheorem{proposition}[theorem]{Proposition}

\numberwithin{theorem}{section}
\numberwithin{equation}{section}

\theoremstyle{definition}
\newtheorem{definition}[theorem]{Definition}
\newtheorem{remark}[theorem]{Remark}
\newtheorem{example}[theorem]{Example}

\newcommand{\from}{\colon}
\newcommand{\R}{{\mathbb{R}}}
\newcommand{\Q}{{\mathbb{Q}}}

\newcommand{\BV}{{\mathrm{BV}}}

\newcommand{\eps}{\varepsilon}
\renewcommand{\epsilon}{\varepsilon}

\newcommand{\norm}[1]{\|#1\|}
\newcommand{\loc}{\mathrm{loc}}
\newcommand{\radon}{\mathscr{M}}

\newcommand{\abs}[1]{|#1|}

\newcommand{\borel}{\mathscr{B}}
\renewcommand{\leq}{\leqslant}

\renewcommand{\geq}{\geqslant}

\renewcommand{\phi}{\varphi}
\newcommand{\relspace}{\hphantom{{}={}}}
\newcommand{\nqquad}{\hspace{-2em}}	
\newcommand{\hf}{{\nicefrac12}}

\DeclareMathOperator*{\esssup}{ess\, sup}
\DeclareMathOperator*{\essinf}{ess\, inf}

\DeclareMathOperator{\sgn}{sgn}
\DeclareMathOperator{\supp}{supp}

\begin{document}

\markboth{Particle paths of hyperbolic conservation laws}{Particle paths of hyperbolic conservation laws}

%
%

\title[Particle paths of hyperbolic conservation laws]{The particle paths of \\ hyperbolic conservation laws}

\author[U. S. Fjordholm]{Ulrik S. Fjordholm}
\address{Department of Mathematics, University of Oslo, PO Box 1053 Blindern, 0316 Oslo, Norway}
\email{ulriksf@math.uio.no}

\author[O. H. Mæhlen]{Ola H. Mæhlen}
\email{olamaeh@math.uio.no}

\author[M. C. Ørke]{Magnus C. Ørke}
\email{magnusco@math.uio.no.}

\begin{abstract}
Nonlinear scalar conservation laws are traditionally viewed as transport equations. We take instead the viewpoint of these PDEs as continuity equations with an implicitly defined velocity field. We show that a weak solution is the entropy solution if and only if the ODE corresponding to its velocity field is well-posed. We also show that the flow of the ODE is $\hf$-H\"older regular. Finally, we give several examples showing that our results are sharp, and we provide explicit computations in the case of a Riemann problem.
\end{abstract}

\maketitle

\section{Introduction}

We consider the scalar conservation law
\begin{equation}\label{eq:cl}
\begin{cases}
\partial_t u + \partial_x f(u) = 0 \\
u(0)=u_0
\end{cases}
\end{equation}
for a given $f \in C^1(\R)$, some $u_0\in \BV_{\loc}\cap L^\infty(\R)$, and an unknown $u=u(x,t)\from\R\times\R_+\to\R$. (Here and elsewhere, we denote $\R_+\coloneqq [0,\infty)$.)

The traditional interpretation of the conservation law \eqref{eq:cl} is as the transport equation
\begin{equation}\label{eq:transport}
    \partial_t u + f'(u)\partial_x u = 0.
\end{equation}
From this formulation, the method of characteristics provides the solution
\begin{equation*}
    u(x, t) = u_0\bigl(\xi_t^{-1}(x)\bigr), \qquad \text{where }\
    \begin{cases}
        \dot{\xi}_t(x) = f'(u(\xi_t,t)) \\
        \xi_0(x) = x,
    \end{cases}
\end{equation*}
(here, we denote $\dot{\xi}=\frac{d\xi}{dt}$, and $\xi_t$ denotes $\xi$ evaluated at time $t$), at least
for smooth initial data $u_0$. However, it is well-known that the method of characteristics might break down at the onset of discontinuities in $u$ (see e.g.~Example \ref{ex:burgers_rarefaction}). A theory of generalized characteristics has been constructed, but at the cost of technical difficulty, see Refs.~\cite{dafermos_1977} and \cite{bianchini_bonicatto_marconi_2021}.

Instead of viewing the conservation law \eqref{eq:cl} as a transport equation \eqref{eq:transport}, we will investigate the connection between \eqref{eq:cl} and the \emph{continuity equation}
\begin{equation}\label{eq:conteq}
    \partial_t v + \partial_x(a v) = 0,
\end{equation}
where $a\from\R\times\R_+\to\R$ is a given vector field. For sufficiently smooth velocity $a$, it is well-known (see e.g.~\cite[Section 2]{ambrosio_crippa_2008}) that the solution of \eqref{eq:conteq} is provided via the ordinary differential equation (ODE)
\begin{equation}\label{eq:flowLinear}
    \begin{cases}
        \dot{x}_t = a(x_t,t) & \text{for}\ t > s,\\
        x_s = x.
    \end{cases}
\end{equation}
If $a$ is, say, Lipschitz continuous then \eqref{eq:flowLinear} is well-posed for all $(x, s) \in \R \times \R_+$. Denote the solution by $x_t=X_t(x,s)$, and for $s=0$, write $X_t(x)=X_t(x,0)$. (The map $X$ is the \emph{flow} of $a$). Then the continuity equation \eqref{eq:conteq} has a unique solution given by the pushforward of the initial data by the flow:
\begin{equation}\label{eq:pushforward}
    v(\cdot,t) = {(X_t)}_{\#} v_0, \quad \text{that is,} \quad \int_{\R} v(x,t)\psi(x)\,dx = \int_{\R} v_0(x) \psi(X_t(x))\,dx
\end{equation}
for all test functions $\psi\in C_c(\R)$. For irregular velocity fields the well-posedness of \eqref{eq:conteq} and \eqref{eq:flowLinear}---and the connection between these---is much more delicate, and remains a subject of active study.

The intuitive interpretation of \eqref{eq:conteq} is that $v$ describes the spatial distribution of \emph{mass}, and that this mass is transported by the velocity field $a$. Thinking of the mass at a single point $x$ at time $s$ as a \emph{particle}, the flow map $t \mapsto X_t(x,s)$ describes the path of that particle forward in time. For this reason, we will refer to these maps as \emph{particle paths}. Our picture is complicated somewhat by the fact that the mass distribution $v$ might attain negative values. We will think of negative values of $v$ as \emph{anti-particles}; as we will see, the correct behaviour at a collision between a particle and an anti-particle is that they merge and that their masses mutually cancel.

Setting $v\coloneqq u-c$ for an arbitrary constant $c\in\R$, it is readily seen that \eqref{eq:cl} is equivalent to the (nonlinear) continuity equation \eqref{eq:conteq} with the velocity field
\begin{equation}\label{eq:velocityfield}
    a_c(x,t) \coloneqq A(u(x,t),c),
    \qquad\text{where}\qquad
    A(u, c) \coloneqq
    \begin{cases}
        \frac{f(u)-f(c)}{u-c}& \text{if } u\neq c,\\
        f'(c) & \text{if } u=c.
    \end{cases}
\end{equation}
It is not at all clear, however, whether the ODE \eqref{eq:flowLinear} corresponding to $a=a_c$ is well-posed; whether the corresponding continuity equation \eqref{eq:conteq} is well-posed; or what this well-posedness implies for the solution $u$ of \eqref{eq:cl}. Since there is no bound on the compression of the flow or the divergence of $a_c$, neither the DiPerna--Lions theory of Lagrangian flows \cite{diperna_lions_1989} nor Ambrosio's theory of $\BV$ vector fields \cite{ambrosio_2004} applies. Our main goal is to show that the well-posedness of these two problems is tightly linked.

\subsection{Main results}
We summarize here our main findings. Their proofs are distributed throughout the remaining sections of the paper.

\begin{theorem}[Main Theorem]\label{thm:main}
Assume $f\in C^1(\R)$ and that $u_0\in \BV_{\loc}\cap L^\infty(\R)$. Let $u\in L^\infty(\R\times\R_+)$ be a weak solution of \eqref{eq:cl} with $u(t)\in \BV_{\loc}(\R)$ for a.e.~$t\geq 0$. Let $a_c$ be defined as in \eqref{eq:velocityfield}. Then the following are equivalent:
\begin{enumerate}
\item $u$ is the entropy solution of \eqref{eq:cl}.
\item The ODE
\begin{equation}\label{eq:flow}
    \begin{cases}
        \dot{x}_t = a_c(x_t,t) & \text{for}\ t > s,\\
        x_s = x,
    \end{cases}
\end{equation}
is well-posed in the Filippov sense for all $x\in\R$, $s\geq0$ and all $c\in\R$.
\end{enumerate}
Moreover, for any $c\in\R$, the entropy solution $u$ has the representation formula
\begin{equation}\label{eq:representation-formula}
    u(\cdot,t) = c + (X^c_t)_\#(u_0-c) \qquad \text{for } t\geq0
\end{equation}
where $X^c=X^c_t(x)$ is the flow of the ODE \eqref{eq:flow}.
\end{theorem}
See Section \ref{sec:preliminaries} for a review of the theory of entropy solutions of conservation laws, Filippov solutions of ODEs, and BV functions. We remark that Filippov solutions of ODEs are solutions in the usual sense wherever the right-hand side is continuous, but that Filippov solutions are also applicable to right-hand sides that are discontinuous or that are only defined almost everywhere.

The proof of sufficiency in the Main Theorem, given in Section \ref{sec:uniquenessOfFlow}, goes via a novel and detailed analysis of particle paths (solutions of the associated ODE \eqref{eq:flow}). As a byproduct of this analysis, we derive a sharp regularity estimate on the flow of the ODE (see Theorem \ref{thm:flowRegular} and Section \ref{sec:flowRegularity}). The proof of necessity, given in Section \ref{sec:selectionPrinciple}, first establishes a novel representation formula for the continuity equation by a commutator estimate (\`a~la~DiPerna--Lions~\cite{diperna_lions_1989}), and then utilizes this to estimate the entropy production of the weak solution.

\newmdenv[innerlinewidth=0.5pt, roundcorner=4pt,linecolor=myLineColor,backgroundcolor=myBackgroundColor,innerleftmargin=6pt,innerrightmargin=6pt,innertopmargin=6pt,innerbottommargin=6pt]{mybox}
\definecolor{myLineColor}{rgb}{0.122, 0.435, 0.698}
\definecolor{myBackgroundColor}{rgb}{0.898, 0.937, 1}
\begin{remark}\label{rem:main_thm_generalization}
The condition $f\in C^1(\R)$ can be generalized to $f$ merely being Lipschitz continuous---at the cost of replacing ``for all $c\in\R$'' with ``for a.e.~$c\in\R$''. Although we will not carry out this generalization here, it is instructive to note the general statement:

\noindent
For any $c\in\R$ where $f$ is differentiable, we have: \\
\begin{minipage}{0.43\linewidth}
\begin{mybox}
$u$ satisfies the Kruzkhov entropy condition for $\eta_c(u)\coloneqq|u-c|$
\end{mybox}
\end{minipage}
\begin{minipage}[c]{0.12\linewidth}
\vspace{-2.5ex}
\begin{align*}
    &\Longrightarrow^* \\
    &\Longleftarrow
\end{align*}
\end{minipage}
\begin{minipage}{0.43\linewidth}
\begin{mybox}
$a_c\coloneqq A(u,c)$ generates a unique Filippov\- flow
\end{mybox}
\end{minipage}

\medskip
The implication $\Longrightarrow$ comes with an additional requirement:
\begin{itemize}
\item[*]
$u$ also satisfies the Kruzkhov entropy condition for constants $c_1^-,c_2^-,\ldots$ converging to $c$ from \emph{below}, as well as constants $c_1^+,c_2^+,\ldots$ converging to $c$ from \emph{above}.
\end{itemize}
Disregarding this technical requirement, the above diagram carries a useful intuition, namely that \emph{any local production of $\eta_c$-entropy corresponds precisely to a local separation---and hence nonuniqueness---of particle paths of the velocity field $a_c$}. (See Fig.~\ref{fig:selection} on page \pageref{fig:selection} for an illustration of this.)
To see why we cannot generally have well-posedness of \eqref{eq:flow} for \emph{all} $c$ when $f$ is Lipschitz, consider the flux $f(u) = |u|$ and initial data $u_0(x)=\sgn(x)$. The corresponding entropy solution $u(x,t)=\frac{1}{2}\big(\sgn(x-t) + \sgn(x+t)\big)$ is zero in the region $|x|<t$; since $f$ is not differentiable at zero, $A(0,0)$ is not well-defined and, thus, neither is $a_0(x,t)$ in the region $|x|<t$. One may attempt to remedy this by defining $A(0,0)\coloneqq r$, for some (natural) choice of $r\in\R$, resulting in the velocity field
\[
a_0(x,t) = \begin{cases}
    -1 & \text{for } x<-t \\
    r & \text{for } x\in(-t,t) \\
    1 & \text{for } x>t.
\end{cases}
\]
However, no matter the value of $r$ we find that both $x_t = -t$ and $x_t = t$ are solutions of \eqref{eq:flow} starting at zero; nonuniqueness is inescapable.
\end{remark}

\begin{remark}\label{rem: BVrequirementsOfU_0}
We restrict our attention to $\BV_\loc(\R)$-valued weak solutions of \eqref{eq:cl}, but this can be slightly weakened: Only the existence of left- and right-limits in space is needed in Section \ref{sec:flowWellPosed} and \ref{sec:selectionPrinciple}.
\end{remark}

Generalizing Theorem \ref{thm:main} to higher dimensions proves difficult:
\begin{example}\label{ex:multiple_dimensions}
The conclusion of the Main Theorem is not true in multiple dimensions without modifications. Indeed, consider
\[
\begin{cases}
    \partial_t u + \nabla_x\cdot f(u) = 0 \\
    u(0)=u_0,
\end{cases}
\qquad
f(u)=\begin{pmatrix}
        u^2\\
        u^2(u^2-1)\end{pmatrix}, \qquad u_0(x_1,x_2)=\mathrm{sgn}(x_2).
\]
The entropy solution is $u(x,t)=u_0(x)$ for all $(x,t)$. Therefore, setting e.g.~$c=0$ yields
\[
a_0(x,t) = \frac{f(u(x,t))}{u(x,t)} = \begin{pmatrix}
    \sgn(x_2) \\ 0
\end{pmatrix}.
\]
The particle paths \eqref{eq:flow} with velocity $a=a_0$ are well-posed for all initial data $x\in\R^2$ \emph{except} when $x_2=0$; in this case any Lipschitz function $x_t=\big(x_t^{(1)}, x_t^{(2)}\big)$ with $x(0)=0$, $x_t^{(2)}=0$ and $|\dot{x}_t^{(1)}|\leq1$ for a.e.~$t>0$ will be a Filippov solution.

We remark that the above $f$ satisfies usual conditions of ``genuine nonlinearity'' (cf.~e.g.~Ref.~\cite{TadmorRascleBagnerini2005}). It is not clear to us what extra conditions, if any, can ameliorate this situation.
\end{example}

We will also prove the following two variants of the Main Theorem, where we arrive at stronger statements using simpler techniques. The proofs are given in Sections \ref{sec:burgers} and \ref{sec:riemann}, respectively.

\begin{theorem}\label{thm:burgers}
    For Burgers' equation the implication \mbox{(i) $\Rightarrow$ (ii)} of the Main Theorem is true under the relaxed condition $u_0\in L^\infty$.
\end{theorem}

\begin{theorem}[Particle paths of Riemann problems]\label{thm:riemann}
Let $f\in C^1(\R)$, let $u_L,u_R\in\R$, and let $u$ be the entropy solution of \eqref{eq:cl} with initial data
\[
u_0(x) = \begin{cases}
u_L & \text{for }x<0\\
u_R & \text{for } x>0.
\end{cases}
\]
Fix $c\in\R$ and consider the ODE \eqref{eq:flow}. Then the solution starting at $(x,s)=(0,0)$ is the straight line $x_t=Vt$, where the velocity $V = V(u_L,u_R)$ is given by
\begin{equation}\label{eq: velocityForRiemannParticle}
V(u_L,u_R)\coloneqq\begin{cases}
\min_{v\in[u_L,u_R]}\tfrac{f(v)-f(c)}{v-c} & \text{if } c<u_L\leq u_R\\
\max_{v\in[u_R,u_L]}\tfrac{f(v)-f(c)}{v-c} & \text{if } c<u_R\leq u_L\\
(f_\smallsmile)'(c) & \text{if } u_L\leq c\leq u_R\vspace{3pt}\\
(f_\smallfrown)'(c) & \text{if }  u_R\leq c\leq u_L\\
\max_{v\in[u_L,u_R]}\frac{f(v)-f(c)}{v-c} & \text{if }  u_L\leq u_R<c\\
\min_{v\in[u_R,u_L]}\tfrac{f(v)-f(c)}{v-c} & \text{if } u_R\leq u_L<c,
    \end{cases}
\end{equation}
where $f_\smallsmile$ and $f_\smallfrown$ are the convex and concave envelopes of $f$ between $u_L$ and $u_R$.
\end{theorem}
 Finally, as a byproduct of the proof of the implication (i) $\Rightarrow$ (ii) in the Main Theorem, we will be able to prove that the flow of $a_c$ is (at least) H\"older regular. A more precise statement and proof of this result is given in Section \ref{sec:flowRegularity} (see Proposition \ref{prop: holderHalfRegularityWithC2Flux} and Remark \ref{rem: extendingRegularityEstimatesToSAndC}).

\begin{theorem}[Regularity of flow]\label{thm:flowRegular}
    Let $f\in C^2(\R)$, let $u$ be the entropy solution of \eqref{eq:cl}, and let $X^c$ be the flow of $a_c$. Then $X^c=X^c_t(x,s)$ is Lipschitz continuous in $t$ and $\hf$-H\"older regular in $x$ and $s$.
\end{theorem}

\subsection{Connections to prior work}
Hyperbolic conservation laws have received an enormous amount of attention from researchers over the past century, but as far as we know, the deep connection to linear continuity equations that we develop in this paper has not been studied before. Some prior works have touched upon ideas developed in this paper, and we try to summarize them here.

Colombo and Marson \cite{Colombo_Marson_2003} study the well-posedness of an ODE arising from the Lighthill--Whitham--Richards model for vehicular traffic. They show that if $u$ is the entropy solution of \eqref{eq:cl} with concave flux, then the ODE $\dot{x}_t = w(u(x_t,t))$ is well-posed (in the Filippov sense) provided $w\from\R\to\R$ satisfies a compatibility condition. In particular, the choice $w(u)=\frac{f(u)}{u}$ becomes a special case of the implication (i) $\Rightarrow$ (ii) of our Main Theorem. They also prove a regularity estimate on the flow of this ODE.

Poupaud and Rascle\cite{PoupaudRascle97} study multidimensional, linear continuity equations with a (possibly discontinuous) velocity field satisfying a one-sided Lipschitz condition. In their Section 5, which is somewhat disconnected from the main part of their paper (the velocity field $a_0$ will generally \emph{not} satisfy any one-sided Lipschitz condition), they show that the particle paths \eqref{eq:flow} for $c=0$ are well-posed for Burgers' equation for the Riemann problem. This is therefore a special case of Theorems \ref{thm:burgers} and \ref{thm:riemann}. Intriguingly, they end their paper by saying that ``{In some sense the above description says that the Burgers equation is \emph{really} a conservation law!}'' We return to a special case of their computation in Example \ref{ex:burgers_rarefaction}.

Variants of the velocity field $a_c \coloneqq \frac{f(u)-f(c)}{u-c}$ and its flow appear in several places in the literature, most notably in Oleinik's seminal works \cite{Oleinik57,Oleinik59}. If $u_1,u_2$ are weak solutions of \eqref{eq:cl} then their difference $v\coloneqq u_1-u_2$ will be a weak solution of the continuity equation
\begin{equation}\label{eq:dualequation}
\partial_t v + \partial_x (bv) = 0, \qquad \text{where } b \coloneqq A(u_1,u_2)= \frac{f(u_1)-f(u_2)}{u_1-u_2}.
\end{equation}
Oleinik used this fact, along with structural assumptions on $u_1,u_2$ (the Oleinik entropy condition), in a duality argument to prove that $u_1,u_2$ coincide whenever their initial data coincide. Note that $u_2\equiv c$ is always a solution of \eqref{eq:cl}, so our velocity $a_c$ is a special case of the above velocity $b$. Oleinik's idea has since been used in several settings: We mention here the generalization to multiple dimensions by Conway and Smoller \cite{ConwaySmoller67} and the duality-based error estimates by Tadmor~et~al.\ (see~Ref.~\cite{Tadmor1991Error} and references therein). It is not clear if our results can be generalized to the velocity field $b$ in \eqref{eq:dualequation}. Indeed, if $u_1=u_2$ in some space-time region then $b=f'(u_1)$, so the particle paths coincide with characteristics; the ill-posedness of characteristics seen for instance in Example \ref{ex:burgers_rarefaction} therefore shows that the particle paths of \eqref{eq:dualequation} might not be well-posed in general.

\subsection{Overview of the paper}
In Section \ref{sec:preliminaries} we give some necessary background information on scalar conservation laws, Filippov solutions of ODEs, and functions of bounded variation. Section \ref{sec:burgers} is dedicated to Burgers' equation; we single out this special case because the special structure of the equation allows for an insightful and simple proof of (part of) the Main Theorem. Similarly, Section \ref{sec:riemann} considers the general Riemann problem, where the special structure of the entropy solution allows for an elementary proof with explicit computations. In Section \ref{sec:flowWellPosed} we prove the sufficiency part of the Main Theorem, as well as the regularity estimate on the flow outlined in Theorem \ref{thm:flowRegular}. In Section \ref{sec:selectionPrinciple} we prove the necessity part of the Main Theorem.

\section{Preliminaries}\label{sec:preliminaries}
We collect here some preliminaries which will be needed later in the paper.

\subsection{Scalar conservation laws}
A function $u\in L^\infty(\R\times\R_+)$ is a \emph{weak solution} of the conservation law \eqref{eq:cl} if it holds in the sense of distributions, that is,
\[
\int_{\R_+}\int_\R u\partial_t\phi + f(u)\partial_x\phi\,dx\,dt + \int_\R u_0(x)\phi(x,0)\,dx = 0
\]
for all $\phi\in C_c^\infty(\R\times\R_+)$. Weak solutions are generally not unique, and selection criteria in the form of \emph{entropy conditions} must be imposed. We say that a pair of functions $(\eta,q)$ is an \emph{entropy pair} if $\eta\from\R\to\R$ is convex and $q\from\R\to\R$ satisfies $q' = \eta'f'$. We say that a weak solution $u$ of \eqref{eq:cl} is an \emph{entropy solution} if
\begin{equation}\label{eq:entropycondition}
    \int_{\R_+}\int_\R \eta(u)\partial_t\phi + q(u)\partial_x\phi\,dx\,dt + \int_\R \eta(u_0(x))\phi(x,0)\,dx \geq 0
\end{equation}
for all $0\leq \phi\in C_c^\infty(\R\times\R_+)$ and all entropy pairs $(\eta, q)$. Kruzkhov \cite{Kru70} proved that there exists a unique entropy solution of \eqref{eq:cl} for any $u_0\in L^\infty(\R)$. If $u_0\in \BV_\loc(\R)$ then also $u(t)\in \BV_\loc(\R)$ for all $t>0$ (see Section \ref{sec:bv-functions}). He showed moreover that it is enough to impose the entropy condition \eqref{eq:entropycondition} for the so-called \emph{Kruzkhov entropy pairs}
\begin{equation}\label{eq:kruzkoventropy}
\eta_k(u) = |u-k|, \qquad q_k(u) = \sgn(u-k)(f(u)-f(k))
\end{equation}
with parameter $k\in\R$. We remark here that we can write
\begin{equation}\label{eq: algebraicIdentitySatisfiedByTheKruzkovEntropies}
     \eta_k(u)a_k=q_k(u),
\end{equation}
where $a_k=A(u,k)$ is as given in \eqref{eq:velocityfield}. This property turns out to be essential in our analysis. We refer to Kruzkhov \cite{Kru70} and the monographs by Holden~and~Risebro\cite{HoldenRisebro} and Dafermos\cite{ Dafermos} for more theory of hyperbolic conservation laws.

\subsection{Filippov solutions}\label{sec:filippov}
For a merely measurable function $a=a(x, t)$, the sense in which the ODE \eqref{eq:flowLinear} (and \eqref{eq:flow}) is understood is not standard. We choose to work with \emph{Filippov solutions}, which are absolutely continuous functions $t\mapsto x_t$ satisfying the differential inclusion
\begin{equation} \label{eq:filippov_ode}
    \underline{a}(x_t, t) \leq \dot{x}_t \leq \overline{a}(x_t, t) \qquad \text{for a.e.~} t,
\end{equation}
where
\begin{equation} \label{eq:lim_ess_sup}
    \underline{a}(x, t) \coloneqq \lim_{\delta \downarrow 0} \essinf_{y \in B_\delta(x)} a(y, t), \qquad
    \overline{a}(x, t) \coloneqq \lim_{\delta \downarrow 0} \esssup_{y \in B_\delta(x)} a(y, t),
\end{equation}
and where $B_\delta(x)\coloneqq (x-\delta, x+\delta)$ (these are the lower and upper semicontinuous envelopes of $a$). Here, as before, $x_t$ denotes the function $x$ evaluated at $t$, and $\dot{x}_t$ denotes the $t$-derivative of this function. This solution concept resolves the challenges of nonexistence and sensitivity to modifications within equivalence classes that typically arise when dealing with ODEs for irregular velocity fields. Indeed, Filippov \cite{Filippov60} proved existence of solutions for any velocity field $a\in L^1_\loc(\R_+; L^\infty(\R))$. Moreover, if the solution is unique, it depends continuously on initial conditions and $L^1_\loc(\R_+; L^\infty(\R))$-perturbations of the right-hand side (see \cite[Theorems 3 and 4]{Filippov60}). We remark that the concept of Filippov solutions is stronger than the ``full solutions'' treated by Conway \cite{conway_1967} in the context of transport equations. Moreover, the concept of Filippov solutions coincides with solutions in the usual sense whenever $a$ is continuous in $x$.

\begin{definition}\label{def:filippovflow}
    Let $a\from\R\times\R_+\to\R$ be measurable. We say that $X$ is a \emph{Filippov flow (generated by $a$)} if $X$ is a measurable function $X=X_t(x,s)$, defined for all $x\in\R$ and $0\leq s\leq t$, such that
    \begin{enumerate}[label=(\roman*)]
        \item $X_s(x,s) = x$ for all $x\in\R$ and $s\geq0$,
        \item for all $x\in\R$ and $s\geq0$, the map $t\mapsto X_t(x,s)$ is absolutely continuous and the derivative $\dot{X}_t(x,s)= \frac{d}{dt}X_t(x,s)$ satisfies
        \begin{equation}\label{eq:filippov_ode_flow}
         \underline{a}(X_t(x,s),t)\leq \dot{X}_t(x,s) \leq \overline{a}(X_t(x,s),t) \qquad \text{for a.e.~} t>s,
        \end{equation}
        that is, $t\mapsto X_t(x,s)$ is a Filippov solution of $\dot{x}=a(x,t)$,
        \item\label{prop:semigroup} $X_t(X_\tau(x, s),\tau) = X_t(x, s)$ for all $x\in\R$ and all $0\leq s\leq \tau \leq t$.
    \end{enumerate}
    If $X$ is also jointly continuous we call it a \emph{continuous Filippov flow}.
\end{definition}

\begin{theorem}\label{thm: preliminariesOnFilippovFlow}
    Let {$a\in L^1_\loc(\R_+; L^\infty(\R))$}. Then
    \begin{enumerate}[label=(\roman*)]
        \item There exists at least one Filippov flow generated by $a$.
        \item The following are equivalent:
        \begin{itemize}
            \item There exists exactly one Filippov flow generated by $a$.
            \item For every $x_s\in\R$ and $s\geq0$, there exists exactly one Filippov solution of $\dot{x}_t=a(x_t,t)$ starting at $(x_s,s)$.
            \item For every $x_0\in\R$ there exists exactly one Filippov solution of $\dot{x}_t=a(x_t,t)$ starting at $(x_0,0)$.
        \end{itemize}
    \item
    If the Filippov flow $X=X_t(x,s)$ is unique, then it is continuous in $(t,x,s)$, and the map $x\mapsto X_t(x,s)$ is surjective and nondecreasing for any $0\leq s\leq t$.
    \end{enumerate}
\end{theorem}

We refer to the excellent paper by Filippov \cite{Filippov60} and the monographs by Filippov\cite{filippov_1988} and H\"ormander \cite{hormander_1997} for details. When the right-hand side of \eqref{eq:flowLinear} is continuous a.e.~in $x$, the following lemma (which we were unable to find in the literature) holds.

\begin{lemma}\label{lemma:convergence_universal_representative}
    Let $a\in L^1_\loc(\R_+; L^\infty(\R))$ be continuous in the $x$ variable almost everywhere. Let $X=X_t(x,s)$ be a continuous Filippov flow generated by $a$. Then
    \begin{equation} \label{eq:convergence_time_difference}
        \lim_{\varepsilon \downarrow 0}  \frac{X_{t+\varepsilon}(x, t) - x}{\varepsilon} = a(x, t) \qquad \text{in}\ L^1_{\loc}(\R\times \R_+).
    \end{equation}
\end{lemma}

\begin{proof}
    Note first that the continuity assumption is equivalent to stating that for some representative of $a$ in its equivalence class, we have that for a.e.~$t>0$, the function $x\mapsto a(x,t)$ is continuous for a.e.~$x\in\R$. Let
    \[
    A \coloneqq \bigl\{(x, t) \in \R \times \R_+ : \dot{X}_t(x, t) \text{ exists and satisfies \eqref{eq:filippov_ode_flow} with } s=t\bigr\},
    \]
   where $\dot{X}_t(x,t)$ should be understood as the forward derivative in $t$ at $s=t$ (since $X$ is not defined for $t<s$). Define further
    \begin{equation*}
        b(x, t) \coloneqq
        \begin{cases}
            \dot{X}_t(x, t) & \text{if } (x, t) \in A \\
            a(x, t) & \text{else}
        \end{cases}
        \qquad \text{for } x\in\R, t\geq0.
    \end{equation*}
    Note that $b$ is everywhere well-defined, that $b\in L^1_\loc(\R_+; L^\infty(\R))$, and that if $\dot{X}_t(x, s)$ exists and satisfies \eqref{eq:filippov_ode_flow} then $(X_t(x, s), t) \in A$, by the semigroup property of the flow (condition \ref{prop:semigroup} of Definition \ref{def:filippovflow}). Define
    \begin{equation*}
        J \coloneqq \bigl\{(x, t) \in \R \times \R_+ : a(\cdot, t) \text{ is discontinuous at } x\bigr\}.
    \end{equation*}
    If $(x,t)\in A\setminus J$ then $\dot{X}_t(x,t)$ exists, equals $b(x,t)$, and lies in $[\underline{a}(X_t(x,t),t),\overline{a}(X_t(x,t),t)] = \{a(X_t(x,t),t)\}$, so necessarily $a=b$ on $A\setminus J$.
    Since $a$ and $b$ can differ only on the set $J \cap A$, and the set $J$ has zero measure in $\R \times \R_+$, it suffices to prove \eqref{eq:convergence_time_difference} for $b$. For any $(x, t) \in \R \times \R_+$ we can write
    \begin{equation} \label{eq:universal_rep}
        \frac{X_{t+\varepsilon}(x, t) - x}{\varepsilon} = \int_0^1 \dot{X}_{t+\varepsilon s}(x, t)\, ds = \int_0^1 b(X_{t + \varepsilon s}(x, t), t + \varepsilon s)\, ds,
    \end{equation}
    since $(X_{t + \varepsilon s}(x, t), t + \varepsilon s) \in A$ for a.e.~$s\in (0, 1)$. This means that
    \begin{equation*}
    \begin{aligned}
        &\nqquad \int_{\R_+} \int_\R \bigg| \frac{X_{t+\varepsilon}(x, t) - x}{\varepsilon} - b(x, t) \bigg|\varphi(x, t)\, dx\, dt \\
        & = \int_{\R_+} \int_\R \bigg| \int_0^1 b(X_{t + \varepsilon s}(x, t), t + \varepsilon s) - b(x, t) \, ds \bigg|\varphi(x, t)\, dx\, dt \\
        & \leq \int_0^1 \int_{\R_+} \int_\R \bigl|b(X_{t}(x, t - \varepsilon s), t) - b(x, t - \varepsilon s)\big|\varphi(x, t - \varepsilon s)\,dx\,dt\,ds
    \end{aligned}
    \end{equation*}
    for all $\varphi \in C_c^\infty(\R \times \R_+)$, where we have used Fubini and a change of variables in the last inequality. Seeing that
    \begin{equation*}
    \begin{aligned}
        \big|b(X_{t}(x, t - \varepsilon s), t) - b(x, t - \varepsilon s)\big| & \leq \big|b(x + o(1), t) - b(x, t)\big| + \big|b(x, t) - b(x, t - \varepsilon s)\big|
    \end{aligned}
    \end{equation*}
    where {we have used the fact that $X_t(x, t-\varepsilon s)-x = o(1)$} as $\eps\downarrow 0$ for all $x,t,s$, the conclusion follows by the continuity of translations in $L^1_\loc(\R\times \R_+)$.
\end{proof}

\subsection{Functions of bounded variation}
\label{sec:bv-functions}

A function $u \in L^1_\loc(\R)$ belongs to $\BV_\loc(\R)$ if its derivative exists as a locally finite Radon measure $Du$, that is
\begin{equation*}
    \int_\R u \partial_x \varphi\, dx = - \int_\R \varphi\, d(Du)
\end{equation*}
for all $\varphi \in C_c^\infty(\R)$. By a locally finite Radon measure on $\R$ we shall mean a measure $\mu$ on the Borel $\sigma$-algebra $\borel(\R)$ which is a finite measure on all compact subsets of $\R$, and we denote by $\radon_{\loc}(\R)$ the space of such measures. If $\mu$ is also a finite measure on $(\R, \borel(\R))$, we drop the local assumption. The total variation measure $\abs{\mu}$ of $\mu \in \radon_\loc$, defined as
\begin{equation*}
    \abs{\mu}(I) = \sup \bigg\{ \int_{\R} \psi(x)\, d\mu : \psi \in C_c(I),\ \norm{\psi}_{C^0} \leq 1 \bigg\}
\end{equation*}
for all open and bounded $I \subset \R$, is a positive, locally finite Radon measure. If $u \in \BV_\loc(\R)$ we always choose to work with a good representative of $u$, that is, a function which minimizes the essential pointwise variation over all functions within its equivalence class. Let $u$ be a good representative and $A \subset I$ the set of atoms of $Du$ on a bounded open interval $I \subset \R$, that is, $A$ comprises points $x \in I$ for which $Du(\{x\}) \neq 0$. Then $A$ is at most countable, $u$ is continuous on $I \setminus A$ and has jump discontinuities at all points in $A$. Furthermore, the derivative has the decomposition
\begin{equation*}
    Du = Du^a + Du^c + \sum_{x \in A} \bigl(u(x+) - u(x-)\bigr) \delta_x,
\end{equation*}
where $Du^a \in L^1_{\loc}(\R)$ is an absolutely continuous measure, $Du^c$ is the Cantor part which is nonatomic, meaning that $Du^c(\{x\}) = 0$ for every $x \in I$, and the last term is a purely atomic measure, with coefficients $u(x\pm) \coloneqq \lim_{h\downarrow 0} u(x\pm h)$, concentrated on $A$. We refer to Ref.~\cite{ambrosio_fusco_pallara_2000} for an overview of the theory for functions of bounded variation.

\section{Particle paths for Burgers' equation}\label{sec:burgers}

Entropy solutions of Burgers' equation with initial data $u_0 \in L^\infty(\R)$ satisfies Oleinik's one-sided estimate \cite{Oleinik57}
\begin{equation}\label{eq:oleinikEC}
    \partial_x u(x, t) \leq \frac{1}{t}
\end{equation}
for all $t > 0$. This special structure implies well-posedness of particle paths through a direct mechanism of uniqueness. The same phenomenon was observed for the Riemann problem in \cite[Section 5]{PoupaudRascle97}, and we include an example from that paper to contrast the well-posedness of particle paths \eqref{eq:flow} to the possible ill-posedness of characteristics.

\begin{example}\label{ex:burgers_rarefaction}
This example is taken from Poupaud and Rascle \cite[Section 5]{PoupaudRascle97}. Consider Burgers' equation with initial data and corresponding entropy solution
    \begin{equation*}
    u_0(x) = \begin{cases}
        0 & \text{if } x<0 \\ 1 & \text{if }x>0,
    \end{cases} \qquad
    u(x,t) =
    \begin{cases}
        0 & \text{if }x<0 \\
        x/t & \text{if }0<x<t \\
        1 & \text{if }x>t.
    \end{cases}
    \end{equation*}
    The characteristic equation is
    \begin{equation*}
        \dot\xi_t = f'(u(\xi_t,t)=u(\xi_t,t), \qquad \xi_0 = x,
    \end{equation*}
    and the equation for particle paths with, say, $c = 0$ takes the form
    \begin{equation} \label{eq:burgers_flow_0}
        \dot{x}_t = a_0(x, t) = \frac{1}{2}u(x_t, t), \qquad x_0 = x.
    \end{equation}
    Since $\xi_t = vt$ is a solution with initial condition $x = 0$ for any $v \in [0, 1]$, the characteristic equation has infinitely many solutions. On the other hand, the particle path with initial condition $x = 0$ is unique and given by $x_t = 0$. In conclusion, the characteristics and the particle paths for this problem are given by
    \begin{equation} \label{eq:burgers_flow_example}
        \Xi_t(x) =
        \begin{cases}
            x & \text{if }x < 0 \\
            vt & \text{if }x = 0 \\
            x + t & \text{if }x > 0,
        \end{cases}
        \qquad X_t(x) =
        \begin{cases}
            x & \text{if }x \leq 0 \\
            \sqrt{2 x t} & \text{if }0 < x \leq \frac{t}{2} \\
            x + t/2 & \text{if }0 < \frac{t}{2} < x
        \end{cases}
    \end{equation}
    for an arbitrary $v \in [0, 1]$, see Fig.~\ref{fig:burgers_ex}.
    \begin{figure}
        \centering
            \includegraphics[width=0.32\textwidth]{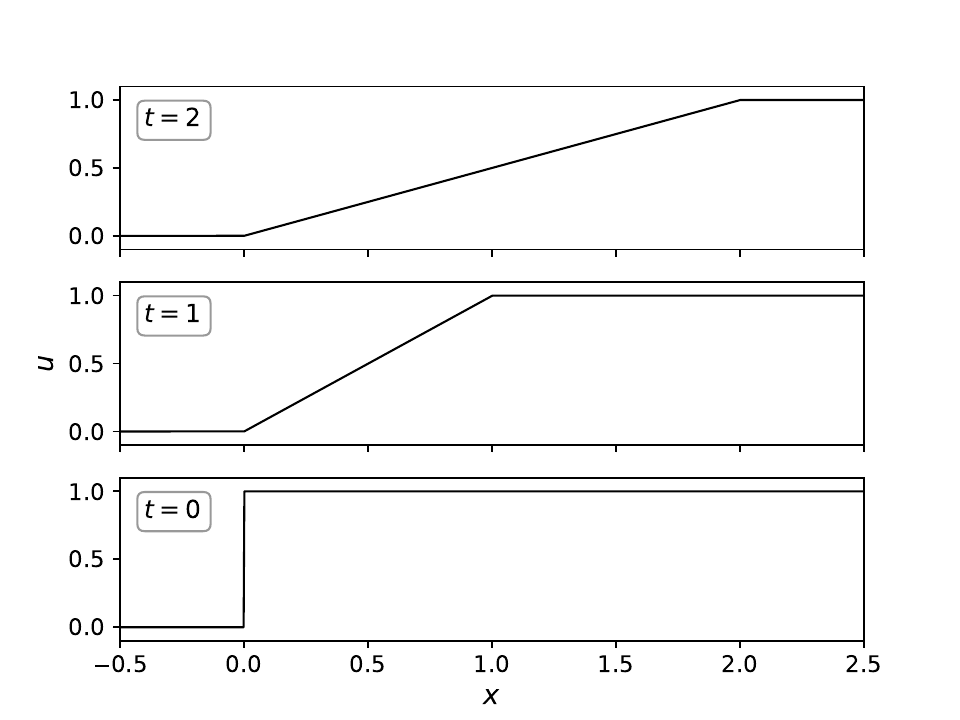}
            \includegraphics[width=0.32\textwidth]{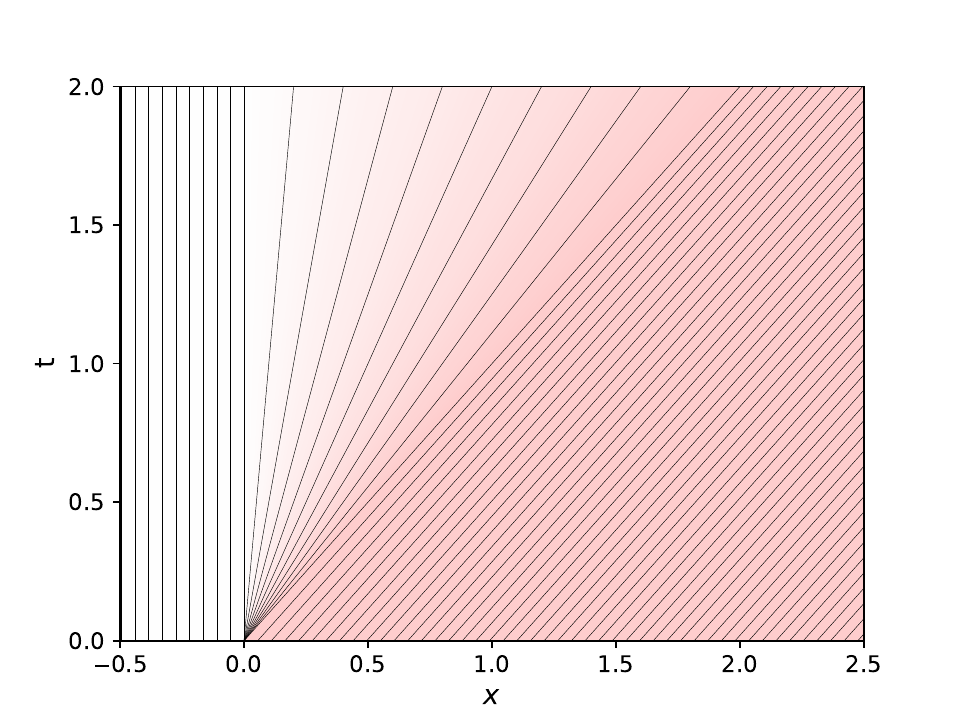}
            \includegraphics[width=0.32\textwidth]{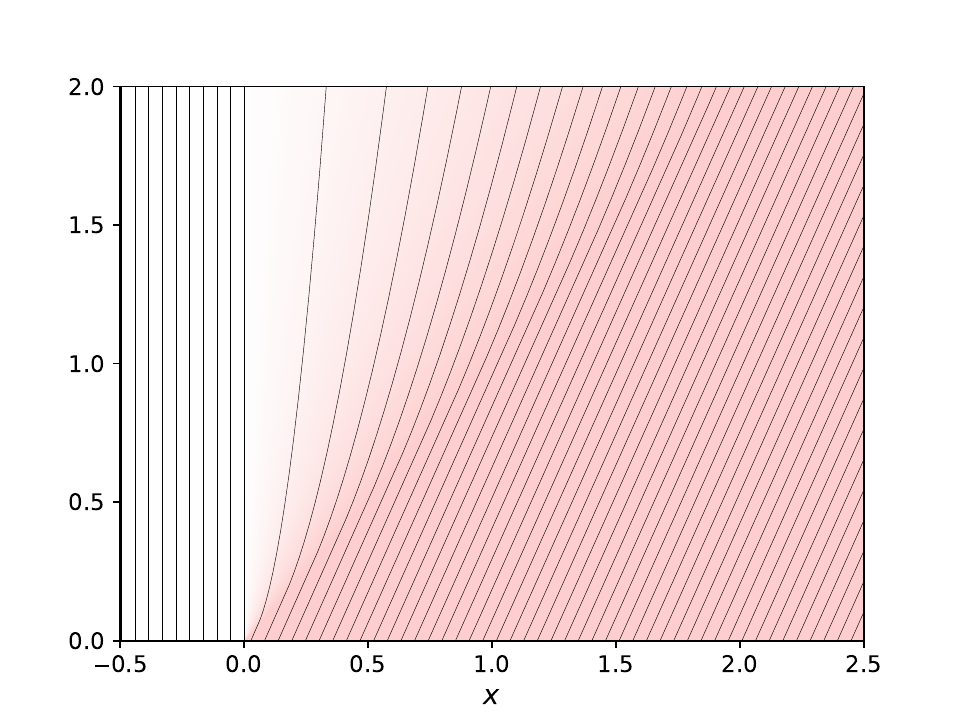}
        \caption{Entropy solution at $t = 0, 1, 2$ (left) for Burgers' equation with Heaviside initial data, characteristics (middle) and particle paths (right). The figures are shaded in red by the magnitude of $u$.}
        \label{fig:burgers_ex}
    \end{figure}
\end{example}

In fact, Oleinik's entropy condition and the factor $\hf$ on the right-hand side of \eqref{eq:burgers_flow_0} completely captures the uniqueness mechanism for all $u_0 \in L^\infty(\R)$ in this case, as demonstrated by the following argument.

\begin{proof}(Proof of Theorem \ref{thm:burgers})
 Let $u$ be the entropy solution of Burgers' equation with initial data $u_0 \in L^\infty(\R)$. Then $u \in L^\infty(\R \times \R_+)$, and existence of a solution to \eqref{eq:flow} is guaranteed by Filippov's existence theorem \cite[Theorem 4]{Filippov60}. Assume that $x_t$ and $y_t$ are two Filippov solutions of \eqref{eq:flow} with initia data $x_0$ and $y_0$. Oleinik's entropy condition \eqref{eq:oleinikEC} yields
    \begin{equation} \label{eq:oleinik}
        \partial_z a_c(z, t) = \partial_z \frac{u(z, t) + c}{2} \leq \frac{1}{2t}
    \end{equation}
    in the sense of distributions for all $t > 0$. Thus, $a_c(\cdot, t)$ is one-sided Lipschitz for all $t > 0$, and assuming that, say, $x_0 \leq y_0$, we may assume $x_t \leq y_t$ without loss of generality. In view of \eqref{eq:oleinik} we get that $a_c(z+,t)\leq a_c(z-,t)$ for all $z$, where the existence of right- and left-limits is justified since $u(\cdot, t) \in \BV_\loc(\R)$ for all $t > 0$. Thus, $x_t$ and $y_t$ satisfy
    \begin{equation*}
         a_c(x_t+, t)\leq \dot{x}_t\leq a_c(x_t-, t), \qquad a_c(y_t+, t)\leq \dot{y}_t \leq a_c(y_t-, t)
    \end{equation*}
    for a.e.~$t > 0$. The inequality \eqref{eq:oleinik} also implies that
    \begin{equation*}
        a_c(y_t-, t) - a_c(x_t+, t) \leq \frac{1}{2t}(y_t - x_t).
    \end{equation*}
    Combining the above, we obtain
    \begin{equation*}
        \frac{d}{dt} (y_t - x_t) \leq a_c(y_t-, t) - a_c(x_t+, t) \leq \frac{1}{2t}(y_t - x_t)
    \end{equation*}
    for a.e.\ $t > 0$. Moreover, boundedness of $u$ implies finite speed of propagation
    \begin{equation*}
        y_t -x_t \leq (y_0 - x_0) + \norm{u_0}_{L^\infty} t,
    \end{equation*}
    and so Grönwall's inequality yields
    \begin{equation*}
        y_t - x_t \leq \sqrt{\frac{t}{t_0}} \bigl(y_{t_0} - x_{t_0}\bigr) \leq \sqrt{\frac{t}{t_0}} \bigl((y_0 - x_0) + \norm{u_0}_{L^\infty} t_0\bigr)
    \end{equation*}
    for all $t > t_0$, where $t_0 > 0$ is an arbitrary parameter. Choosing ${t_0 = (y_0 - x_0)/\norm{u_0}_{L^\infty}}$ in the inequalities above, we arrive at the upper bound
    \begin{equation}\label{eq:burgersHolderEstimate}
        \abs{y_t - x_t} \leq \abs{y_0 - x_0} + 2 \sqrt{\norm{u_0}_{L^\infty} \abs{y_0 - x_0} t}
    \end{equation}
    for all $t > 0$. In particular, if $x_0 = y_0$ then $x_t = y_t$ for all $t\in\R_+$.
\end{proof}

\begin{remark}
    The proof of Theorem \ref{thm:burgers} also applies to entropy solutions of \eqref{eq:cl} with monomial flux functions $f(u) = \abs{u}^p$ for $p>1$ and $c = 0$ (but unfortunately not for $c\neq0$). Indeed, Hoff's Oleinik-type estimate $\partial_x f'(u(x, t)) \leq \frac{1}{t}$ for entropy solutions with convex flux functions \cite{hoff_1983} implies
    \begin{equation*}
        \partial_x a_0(x, t) = \partial_x \frac{1}{p} f'(u(x, t)) \leq \frac{1}{pt},
    \end{equation*}
    and consequently
    \begin{equation*}
        \frac{d}{dt} (y_t - x_t) \leq \frac{1}{pt} (y_t - x_t)
    \end{equation*}
    for particle paths $x_t \leq y_t$. Since $p > 1$ the same approach applies.
\end{remark}

\section{The Riemann problem}\label{sec:riemann}
This section is devoted to the proof of Theorem \ref{thm:riemann}. As is often the case for results on solutions to Riemann problems, the proof is a tedious, but elementary, case-by-case analysis.

\begin{proof}(Proof of Theorem \ref{thm:riemann})
If $u_L=u_R$ the result follows trivially and so we consider the case when $u_L< u_R$; we skip the case when $u_L>u_R$ since a change of variables would move this over to the considered setting. Let $f_{\smallsmile}$ be the lower convex envelope of $f$ on $[u_L,u_R]$, that is, for $v\in[u_L,u_R]$ we define
\begin{align}\label{eq: definitionOfConvexEnvelope}
    f_\smallsmile(v) = \max\bigl\{h(v): h \text{ convex and } h\leq f \text{ on }[u_L,u_R]  \bigr\}.
\end{align}
It is straightforward to show that $f_\smallsmile$ is itself convex. Since $f_{\smallsmile}'\coloneqq \frac{d}{du}f_{\smallsmile}$ is continuous (as $f$ is $C^1$) and nondecreasing it admits right-inverses: Let $g_L$ and $g_R$ denote the left- and right-continuous functions
\begin{align*}
    g_L(y)=\begin{cases}
        u_L, \hspace{53pt}y<y_L,\\
        \min (f_{\smallsmile}')^{-1}(y) ,\,\, y\in[y_L, y_R],\\
        u_R,\hspace{53pt} y>y_R,
    \end{cases} \quad
    g_R(y)=\begin{cases}
        u_L,\hspace{53pt} y<y_L,\\
        \max (f_{\smallsmile}')^{-1}(y) ,\,\, y\in[y_L, y_R],\\
        u_R, \hspace{53pt}y>y_R,
    \end{cases}
\end{align*}
where $y_L\coloneqq f_{\smallsmile}'(u_L)$ and $y_R\coloneqq f_{\smallsmile}'(u_R)$ and where the set $(f_{\smallsmile}')^{-1}(y)$ is the preimage of $f_{\smallsmile}'$ at $y$. We now point out two facts: First, the entropy solution of the Riemann problem is characterized by
    \begin{align*}
        u(x-,t) =g_L\Big(\frac{x}{t}\Big), \qquad
         u(x+,t) =g_R\Big(\frac{x}{t}\Big).
    \end{align*}
Second, we have the two identities
\begin{align}\label{eq: identitiesForInverses}
   f\circ g_L=  f_{\smallsmile}\circ g_L , \qquad  f\circ g_R=f_{\smallsmile}\circ g_R,
\end{align}
whose proof we skip for the sake of brevity; it mostly boils down to definitions.

Next, with $V$ as defined in \eqref{eq: velocityForRiemannParticle} and $A$ as in \eqref{eq:velocityfield} we aim to show that
\begin{align}\label{eq: RiemannVelocitySqueeze}
      A(g_R(V), c) \leq  V \leq A(g_L(V), c).
\end{align}
Using the above expressions for $u$, the previous two inequalities are equivalent to the bounds $a_c(x_t+,t)\leq \dot{x}_t\leq a_c(x_t -, t)$ where $a_c=A(u,c)$ and where $x_t = Vt$. In particular, the theorem would follow as $x_t$ is then a Filippov solution of \eqref{eq:flow}. We break the remainder of the proof up into three cases.

\textit{The case when $c< u_L < u_R$.} Since $V\coloneqq \min_{v\in[u_L,u_R]}A(v,c)$, the second inequality in \eqref{eq: RiemannVelocitySqueeze} holds trivially, and so we turn to the first inequality. For $v\in [u_L,u_R]$ consider the affine function
\begin{equation*}
    h(v)\coloneqq V(v-c) + f(c)\leq A(v,c)(v-c)+f(c)=f(v).
\end{equation*}
Since $h$ in particular is convex, we see that $h\leq f_\smallsmile\leq f$ on $[u_L,u_R]$. Pick $\overline{u}\in[u_L,u_R]$ such that $V=A(\overline{u},c)$ and note that $h(\overline{u})=f(\overline{u})$. Thus, $h\leq f_\smallsmile$ on $[u_L,u_R]$ with $h(\overline{u})=f_\smallsmile(\overline{u})$, and so we infer that
\begin{align*}
\begin{cases}
    V\leq f_\smallsmile'(\overline{u}) & \text{if }\overline{u}=u_L,\\
    V=f_\smallsmile'(\overline{u}) & \text{if }\overline{u}\in(u_L,u_R),\\
    V\geq f_\smallsmile'(\overline{u}) & \text{if }\overline{u}=u_R,
\end{cases}
\end{align*}
where we used $h'=V$. The above relations tells us that if $ f_\smallsmile'(\overline{u})\neq V$, then either $\overline{u}=u_L$ and $V<y_L\coloneqq f_\smallsmile'(u_L)$, or $\overline{u}=u_R$ and $V>y_R\coloneqq f_\smallsmile'(u_R)$; in either case it is clear from the definition of $g_R$ that $g_R(V)=\overline{u}$ which implies the first inequality in \eqref{eq: RiemannVelocitySqueeze}. Thus, it remains to cover the case $f_\smallsmile'(\overline{u})= V $: Again, by the definition of $g_R$, and the fact that $f_\smallsmile'$ is nondecreasing, we infer $\overline{u}\leq g_R(V)$ and $f_\smallsmile'(v)=V$ for all $v\in [\overline{u},g_R(V)]$. Combining this with the identity $f\circ g_R=f_\smallsmile\circ g_R $ \eqref{eq: identitiesForInverses} and the properties of $\overline{u}$ we compute
\begin{align*}
   \big(g_R(V)-c\big) A(g_R(V),c) =&\, f(g_R(V))-f(c)\\
   =&\,\big(f_\smallsmile(g_R(V))-f_\smallsmile(\overline{u})\big) + \big(f(\overline{u})-f(c)\big)\\
   =&\, (g_R(V)-\overline{u})V + (\overline{u}-c)V = (g_R(V)-c)V.
\end{align*}
Since $c\neq g_R(V)$, we obtain the first inequality in \eqref{eq: RiemannVelocitySqueeze} as an equality.

\textit{The case when $u_L < u_R < c$.}
This case can be treated by an analogous argument or, alternatively, by a change of variables $x\mapsto-x$, $u\mapsto -u$, $c \mapsto -c$ and $f\mapsto f(-\cdot)$ which maps this case into the previous one.

\textit{The case when $u_L \leq c \leq u_R$.}
Since $V \coloneqq f_{\smallsmile}'(c)$ and $c\in[u_L,u_R]$, the definitions of $g_L$, $g_R$ give both that $g_L(V)\leq c\leq g_R(V)$ and that $f_\smallsmile' = V$ on $[g_L(V),g_R(V)]$.
We begin by demonstrating the first inequality in \eqref{eq: RiemannVelocitySqueeze}, which requires us to consider three sub-cases. Suppose first that $c<g_R(V)$:
Combining $f( g_R(V))=f_\smallsmile( g_R(V))$ (due to \eqref{eq: identitiesForInverses}) with $f_\smallsmile(c)\leq f(c)$ gives
\begin{align*}
    A(g_R(V),c) \leq \frac{f_{\smallsmile}(g_R(V))-f_\smallsmile(c)}{g_R(V)-c} = V,
\end{align*}
where the equality follows from the fact that $f_\smallsmile'=V$ on $[c,g_R(V)]$. Next, suppose $c=g_R(V)<u_R$: As $g_R$ is right-continuous and strictly increasing on $[V,y_R]$ we infer from the identity $f\circ g_R=f_\smallsmile\circ g_R$ that
\begin{align*}
    f'(c)=\lim_{h\downarrow0}\frac{f(g_R(V+h)) - f(c)}{g_R(V+h)-c}=\lim_{h\downarrow0}\frac{f_{\smallsmile}(g_R(V+h)) - f_{\smallsmile}(c)}{g_R(V+h)-c}=f_\smallsmile'(c)
\end{align*}
or, equivalently, $A(g_R(V),c)=V$. Finally, suppose $u_L<c=g_R(V)$: Since $f_\smallsmile\leq f$ on $[u_L,c]$ and $f_\smallsmile(c)=f(c)$ we conclude that
\begin{align*}
    f'(c)\leq f_{\smallsmile}'(c) \quad \Longleftrightarrow \quad A(g_R(V),c)\leq V.
\end{align*}
From these three sub-cases we conclude that the first inequality in \eqref{eq: RiemannVelocitySqueeze} holds true (the case $u_L=c=u_R$ is excluded by the assumption $u_L<u_R$). And by considering three similar sub-cases, the same can be said about the second inequality.
\end{proof}

\section{Well-posedness of particle paths}\label{sec:flowWellPosed}

In this section we prove that if $u$ is the entropy solution of \eqref{eq:cl} then the ODE \eqref{eq:flow} is well-posed in the Filippov sense (see Section \ref{sec:filippov}) for all initial data $(x,s)\in \R\times \R_+$ and all $c\in\R$. By Theorem \ref{thm: preliminariesOnFilippovFlow} there always exists a Filippov solution, so we need only prove the uniqueness of Filippov solutions to \eqref{eq:flow}; this is done in Section \ref{sec:uniquenessOfFlow}. Theorem \ref{thm: preliminariesOnFilippovFlow} further tells us that the well-posedness of \eqref{eq:flow} gives, for each $c\in\R$, rise to a unique and continuous Filippov flow $X^c_t(x,s)$ (see Definition \ref{def:filippovflow}). The regularity of this mapping is the subject of Section \ref{sec:flowRegularity}.

\subsection{Uniqueness of the flow}\label{sec:uniquenessOfFlow}
In this section we prove that the ODE \eqref{eq:flow} has at most one Filippov solution when $u$ is the entropy solution of \eqref{eq:cl}. It should be mentioned that the computations in this subsection can be carried out in a more general framework than that of Theorem \ref{thm:main}: For fixed $c\in \R$ the well-posedness of \eqref{eq:flow} requires only that
\begin{enumerate}[label=(\roman*)]
\item $f\in C(\R)$ is differentiable at $c$,
\item $u\in L^\infty(\R\times\R_+)$ admits left- and right-limits in space for a.e. $t>0$, and
\item $c$ is an accumulation point, from left and right, in the set of $k\in \R$ such that $u$ satisfies \eqref{eq:entropycondition} with entropy $\eta_k(u)=|u-k|$.
\end{enumerate}
In particular, this proves the implication $\implies$ from Remark \ref{rem:main_thm_generalization}. Condition (iii), which is stronger than the left box in Remark \ref{rem:main_thm_generalization}, is needed at the very end of the following proposition; its importance is discussed in Remark \ref{rem: theCorrectKIsTheOneCuttingThroughTheShock}.

\begin{proposition}[Velocity compression]\label{prop: velocityControl} Let $f\in C^1(\R)$ and $u$ be the entropy solution of \eqref{eq:cl} with $u_0\in L^\infty\cap \BV_{\loc}(\R)$. Fix $c\in\R$ and let further $x_t$ be a {Filippov solution of \eqref{eq:flow}}. Then
\begin{align*}
    a_c({x_t+},t)\leq \dot{x}_t\leq a_c({x_t-},t),
\end{align*}
for a.e.~$t>s$.
\end{proposition}
\begin{proof}
Without loss of generality we set $s=0$, as we could otherwise make a change of variables. Let $0\leq \Lambda\in C^\infty_c(\R)$ be a smooth hat function, that is, let $\Lambda$ be symmetric about zero, nondecreasing on $(-\infty,0]$ and satisfy $\Lambda(0)=1$. For $\eps>0$ define
\begin{equation*}
    \Lambda_\eps(z,t) = \Lambda\Big(\frac{z-x_t}{\eps}\Big),
\end{equation*}
and observe that we for all $t\geq 0$ and $v\in \BV_{\loc}(\R)$ have
\begin{align}\label{eq: lambdaConvergesToDifferenceOfLeftAnd>RightLimit}
    \lim_{\eps\downarrow 0}\int_{\R} \partial_z\Lambda_\eps(z,t)v(z) \,dz = v({x_t-})-v({x_t+}).
\end{align}
For any $0\leq \theta\in C_c^\infty((0,\infty))$ the function $(z,t)\mapsto \Lambda_\eps(z,t)\theta(t)$ is a nonnegative element of $W^{1,1}(\R\times\R_+)$ and is thus an admissible test function for the entropy inequalities satisfied by $u$. In particular, using Kruzkhov's entropies \eqref{eq:kruzkoventropy} we get
\begin{equation} \label{eq: generalEntropyInequalityWithLambdaFunction}
    \begin{split}
    0&\leq \int_{\R_+}\int_\R\eta_k(u)\partial_t(\Lambda_\eps\theta) + q_k(u)\partial_z(\Lambda_\eps\theta) \,dz \,dt\\
    &=\int_{\R_+}\int_\R\eta_k(u)\Lambda_\eps \partial_t\theta\,dz \,dt\\
   &\relspace +  \int_{\R_+}\int_\R \Big(q_k(u)-\eta_k(u)\dot{x}_t\Big)\partial_z\Lambda_\eps\theta \,dz \, dt,
    \end{split}
\end{equation}
for all $k\in\R$, where we exploited the identity $\partial_t \Lambda_\eps=-\dot{x}_t\partial_z \Lambda_\eps$ which holds a.e.\ in $\R\times\R_+$.
Clearly the first integral on the right-hand side of \eqref{eq: generalEntropyInequalityWithLambdaFunction} vanishes as $\eps\downarrow 0$ whereas the limit of the second is implied by \eqref{eq: lambdaConvergesToDifferenceOfLeftAnd>RightLimit} and dominated convergence. Letting $\eps$ tend to zero in \eqref{eq: generalEntropyInequalityWithLambdaFunction} we thus obtain
\begin{align*}
     0\leq&\,   \int_{\R_+}\Big(\big\llbracket \eta_k(u)\big\rrbracket (t)\dot{x}_t- \big\llbracket q_k(u) \big\rrbracket(t) \Big)\theta(t) \,dt,
\end{align*}
where we introduced the notation $\llbracket g\rrbracket(t) \coloneqq g({x_t+},t)-g({x_t-},t)$. By generality of $\theta$ we further conclude for any (fixed) $k\in\R$ that
\begin{align}\label{eq: entropyInequalityAlongAPath}
    0\leq \big\llbracket \eta_k(u)\big\rrbracket \dot{x}_t- \big\llbracket q_k(u) \big\rrbracket,
\end{align}
for a.e.\ $t\geq 0$. Let $E_k$ denote the null set of points $t\geq 0$ where \eqref{eq: entropyInequalityAlongAPath} fails. Wishing to allow $k$ in \eqref{eq: entropyInequalityAlongAPath} to vary for a fixed $t$, we introduce the null set $E=E_c\cup(\cup_{k\in \Q}E_k)$ and get that \eqref{eq: entropyInequalityAlongAPath} holds for all $(k,t)\in (\Q\cup \{c\})\times(\R_+\setminus E)$.

Exploiting the identities $\eta_k(u)=|u-k|$ and $q_k(u)=|u-k|a_k$ in \eqref{eq: entropyInequalityAlongAPath}, and using the notation $u^{\pm}(t)=u({x_t\pm},t)$ and $a_k^{\pm}(t)=a_k({x_t\pm},t)$, we obtain
\begin{align}\label{eq: rewritingEntropyInequalityAlongAPath}
    0\leq  \bigl|u^+(t)-k\bigr|\big(\dot{x}_t - a_k^+(t)\big) + \big|u^-(t)-k\big|\big(a_k^-(t) -\dot{x}_t\big)
\end{align}
for all $(k,t)\in (\Q\cup \{c\})\times(\R_+\setminus E)$. As $x_t$ is a Filippov solution of the ODE induced by $a_c$, it follows that $\dot{x}_t\in [a_c^-(t)\wedge a_c^+(t),a_c^-(t)\vee a_c^+(t)]$ almost everywhere; by possibly removing another null set, we may assume this inclusion to hold for all $t\in \R_+\setminus E$. The proposition is then proved if we can show that $a^+_c(t)\leq a^-_c(t)$ in $\R_+\setminus E$. This is where the regularity of $f$ comes in; we will need its derivative at $c$.

Assume for the sake of contradiction that there is some $t\in \R_+\setminus E$ such that $a^-_c(t)<a^+_c(t)$. Observe then that we can not have $u^+(t)=u^-(t)$ as this would force $a_c^+(t)=a_c^-(t)$. Setting $k=c$ in \eqref{eq: rewritingEntropyInequalityAlongAPath} we conclude that either $u^+(t)=c$ and $\dot{x}_t=a^-_c(t)$, or $u^-(t)=c$ and $\dot{x}_t=a^+_c(t)$. We proceed with the former alternative; the latter is dealt with similarly and so we skip it. For brevity, we shall omit the (fixed) argument $t$ from here on. Recall the definition \eqref{eq:velocityfield} of $A$ and observe that $k\mapsto A(v,k)$ is continuous at $k=v$ and differentiable for $k\neq v$.  From $u^+=c\neq u^-$, we conclude that the two functions $k\mapsto a^+_k =A(u^+,k)$ and $k\mapsto a_k^-=A(u^-,k)$ are respectively continuous and differentiable at $k=c$. Exploiting this, and the identities $u^+=c$ and $\dot{x}_t =a_c^-$, the two terms on the right-hand side of \eqref{eq: rewritingEntropyInequalityAlongAPath}  can be written
\begin{align*}
     \big|u^+-k\big|\big(\dot{x}_t - a_k^+\big) &=|c-k|\big(a_c^-- a_c^+\big) + o(c-k),\\
    \big|u^--k\big|\big(a_k^- -\dot{x}_t\big) &= \big|u^--c\big|\tfrac{\partial A}{\partial k}(u^-,c)(k-c) + o(c-k)
\end{align*}
as $k\to c$. Substituting this into \eqref{eq: rewritingEntropyInequalityAlongAPath} yields
\begin{align}\label{eq: concludingTheVelocityControlProposition}
    0\leq |c-k|\Big(a_c^-- a_c^+\  + |u^--c|\tfrac{\partial A}{\partial k}(u^-,c)\sgn(k-c)\Big) + o(c-k),
\end{align}
for all $k\in \Q\cup\{c\}$. The parentheses is bounded above by $a_c^--a_c^+<0$ when we restrict either $k>c$ or $k<c$, depending on the sign of $\tfrac{\partial A}{\partial k}(u^-,c)$. Thus, letting $k\in \Q$ tend to $c$ (from the appropriate side of $c$) we eventually obtain a strictly negative right-hand side in \eqref{eq: concludingTheVelocityControlProposition}. By contradiction, we conclude that $a_c^+(t)\leq a_c^-(t)$ for all $t\in\R_+\setminus E$ and so we are done.
\end{proof}
\begin{remark}\label{rem: theCorrectKIsTheOneCuttingThroughTheShock}
    Exploiting the identity $\tfrac{\partial A}{\partial k}(u,k)= (A(u,k)-f'(k))/(u-k)$ at the very end of the previous proof, one finds that ``the appropriate side of $c$'' is the one such that $k$ intersects the shock connecting $u^-$ and $u^+$. Informally, this means that Proposition \ref{prop: velocityControl} (and the well-posedness of \eqref{eq:flow}) follows as long as $c$ only intersects entropic shocks in $u$. This condition is similar, though stronger, to requiring that $u$ satisfies \eqref{eq:entropycondition} with entropy $\eta_c(u)=|u-c|$, as the latter allows $c$ to intersect the end-points of nonentropic shocks in $u$. To exemplify, consider the weak solution $u(x,t) = 1+ \sgn(x-t)$ of Burgers' equation. Then $u$ satisfies \eqref{eq:entropycondition} with entropy $\eta_0(u)=|u|$, but the ODE \eqref{eq:flow} is \textit{not} well-posed for $c=0$ because $x_t=0$ and $x_t=t$ are both solutions starting at zero.
\end{remark}

Proposition \ref{prop: velocityControl} allows us to compute explicit bounds on entropies of $u$ in between particles. This is Lemma \ref{lemma: explicitLocalEntropyControl}, of which a special case is the following result, that can be interpreted as saying that the total mass between two particles can not increase over time.

\begin{proposition}[Local entropy dissipation]\label{prop: localMassDecay}
Let $f\in C^1(\R)$ and $u$ be the entropy solution of \eqref{eq:cl} with $u_0\in L^\infty\cap \BV_{\loc}(\R)$. Fix $c\in\R$ and let further $x_t\leq y_t$ be Filippov solutions of \eqref{eq:flow}. Then the function
\begin{align*}
    t\mapsto \int_{x_t}^{y_t}|u(z,t)-c|\,dz
\end{align*}
is nonincreasing.
\end{proposition}
\begin{proof}
    This follows directly from Lemma \ref{lemma: explicitLocalEntropyControl} applied to the Kruzkhov entropy pair \eqref{eq:kruzkoventropy} with $k=c$ because of the identity \eqref{eq: algebraicIdentitySatisfiedByTheKruzkovEntropies}.
\end{proof}
Before stating Lemma \ref{lemma: explicitLocalEntropyControl}, we first point out that the two propositions presented so far suffice to prove the first part of Theorem \ref{thm:main}.

\begin{proof}(Proof of (i) $\Rightarrow$ (ii) in Theorem \ref{thm:main})
    As explained in the introduction of the section, we need only prove the uniqueness of paths. Fix $c\in \R$ and let $x_t$ and $y_t$ be Filippov solutions of the ODE \eqref{eq:flow} such that $x_s=y_s$. Since $\tilde{x}_t \coloneqq x_t\wedge y_t$ and $\tilde{y}_t \coloneqq x_t\vee y_t$ are also particle paths, we may assume without loss of generality that $x_t\leq y_t$ for all $t\geq s$. Let $D$ denote the set of times $t>s$ such that $x_t<y_t$. It follows from Proposition \ref{prop: localMassDecay} that $u=c$ almost everywhere between  $x_t$ and $y_t$ for every $t\in D$. By Proposition \ref{prop: velocityControl} we then infer
    \begin{align*}
        \dot{y}_t-\dot{x}_t\leq a_c({y_t-},t)-a_c({x_t+},t)=f'(c)-f'(c)=0,
    \end{align*}
   for a.e.\ $t\in D$. Since $x_t=y_t$ for $t\notin D$, in particular at $t=s$, we conclude that $x_t = y_t$ for all $t\geq s$.
\end{proof}
We end the subsection with Lemma \ref{lemma: explicitLocalEntropyControl}, whose generality will be important when we in the next subsection examine the spatial regularity of the flow of \eqref{eq:flow}.

\begin{lemma}[Bound on local entropy growth]\label{lemma: explicitLocalEntropyControl}
Let $f\in C^1$ and $u$ be the entropy solution of \eqref{eq:cl} with $u_0\in L^\infty\cap \BV_{\loc}(\R)$. Fix $c\in \R$ and let further $x_t\leq y_t$ be Filippov solutions of \eqref{eq:flow} with initial time $s\geq0$. For any entropy pair $(\eta,q)$ with $\eta\geq 0$, the corresponding $\eta$-entropy of $u$ between $x_t$ and $y_t$ admits the growth bound
    \begin{align*}
    \frac{d}{ \,dt}\int_{x_t}^{y_t}\eta(u(z,t)) \,dz\leq&\,  \Big[\eta\big(u(y_t-,t)\big)a_c(y_t-,t) - q\big(u(y_t-,t)\big)\Big]\\
    &\,-\Big[\eta\big(u(x_t+,t)\big)a_c(x_t+,t) - q\big(u(x_t+,t)\big)\Big]
    \end{align*}
for a.e.~$t>s$.
\end{lemma}
\begin{proof}
Without loss of generality, set $s=0$. Let $0\leq \chi\in C^\infty(\R)$ be nondecreasing and satisfy $\chi=0$ on $(-\infty,0]$ and $\chi=1$ on $[1,\infty)$. For $\varepsilon>0$, define
\begin{align}\label{eq: smoothStepFunction}
    \chi_\eps^x(z,t) = \chi\Big(\frac{z-x_t}{\eps}\Big), \qquad \chi_\eps^y(z,t) = \chi\Big(\frac{y_t-z}{\eps}\Big) \qquad \text{for } z\in\R,\ t\geq0.
\end{align}
Finally, let $0\leq \theta\in C^\infty_c\big((0,\infty)\big)$ and note that the test function $\varphi(z,t)\coloneqq \chi_\eps^x(z,t)\chi_\eps^y(z,t)\theta(t)$ is a nonnegative element of $W^{1,1}(\R\times\R_+)$. By the entropy inequalities \eqref{eq:entropycondition} we find
\begin{equation}\label{eq: entropyInequalityForParticlePathsOfRiemannSolution}
    \begin{split}
    0 &\leq \int_{\R_+}\int_\R\eta(u)\partial_t\varphi + q(u)\partial_z\varphi \,dz \,dt\\
    &= \int_{\R_+}\int_\R\eta(u)\chi_\eps^x\chi_\eps^y \,dz \,\theta'(t)\,dt\\
    &\relspace +\int_{\R_+}\int_\R\Big[\eta(u)\partial_t\chi_\eps^x+q(u)\partial_z\chi_\eps^x\Big]\chi_\eps^y \,dz\,\theta(t)\,dt\\
    &\relspace +\int_{\R_+}\int_\R\Big[\eta(u)\partial_t\chi_\eps^y+q(u)\partial_z\chi_\eps^y\Big]\chi_\eps^x \,dz\,\theta(t)\,dt.
    \end{split}
\end{equation}
We may for the last two integrals utilize the pair of identities $\partial_t\chi_\eps^x=-\dot{x}_t\partial_z\chi_\eps^x$ and $\partial_t\chi_\eps^y=-\dot{y}_t\partial_z\chi_\eps^y$ which hold a.e.~~in $\R\times(0,\infty)$. Moreover, similar to the limit \eqref{eq: lambdaConvergesToDifferenceOfLeftAnd>RightLimit}, we here find for a.e.~~$t\geq 0$ that
\begin{equation*}
    \begin{split}
\int_\R\Big[-\eta(u)\dot{x}_t+q(u)\Big]\big(\partial_z\chi_\eps^x\big)\chi_\eps^y \,dz &\to -\eta(u(x_t+,t))\dot{x}_t + q(u(x_t+,t)),\\
\int_\R\Big[-\eta(u)\dot{y}_t+q(u)\Big]\big(\partial_z\chi_\eps^y\big)\chi_\eps^x\,dz &\to \eta(u(y_t-,t))\dot{y}_t - q(u(y_t-,t)),
    \end{split}
\end{equation*}
as $\eps\downarrow 0$. By dominated convergence, these limits carry over to \eqref{eq: entropyInequalityForParticlePathsOfRiemannSolution} giving
\begin{align*}
     0\leq &\,\int_{\R_+}\Bigg(\int_{x_t}^{y_t}\eta(u)\,dz\Bigg) \theta'(t)\,dt\\  &\,+\int_{\R_+}\Big[\eta(u(y_t-,t))\dot{y}_t - q(u(y_t-,t))\Big]\theta(t)\,d t\\
    &\, -\int_{\R_+} \Big[\eta(u(x_t+,t))\dot{x}_t - q(u(x_t+,t))\Big]\theta(t)d t.
\end{align*}
The result then follows from the inequalities $\dot{y}_t\leq a_c(y_t-,t)$ and $\dot{x}_t\geq a_c(x_t+,t)$
provided by Proposition \ref{prop: velocityControl}.
\end{proof}

\subsection{Spatial regularity of the flow}\label{sec:flowRegularity}
We here examine the regularity of the Filippov flow $X^c=X^c_t(x,s)$ corresponding to the ODE \eqref{eq:flow} when $u$ is the entropy solution of \eqref{eq:cl}. As \eqref{eq:flow} is well-posed in the Filippov sense for every $c\in\R$, Theorem \ref{thm: preliminariesOnFilippovFlow} guarantees both the existence and uniqueness of the flow and its continuity in $x$,~$t$~and~$s$. In fact, the flow is continuous in $c$ as well: By Filippov's stability theorem \cite[Theorem 3]{Filippov60}, the $c$-continuity of the velocity field $a_c$ carries over to the flow. Furthermore, it is clear from its definition that $X^c_t(x,s)$ is Lipschitz continuous in $t$, but its regularity in $x$,~$s$~and~$c$ (other than continuity) is less apparent. We will here focus on the spatial regularity of the flow; Remark \ref{rem: extendingRegularityEstimatesToSAndC} explains how to infer similar estimates in $s$ and $c$.

As $x\mapsto X_t^c(x,s)$ is nondecreasing, our goal will be to find bounds on its growth. This growth reflects the separation of particles: A rapid growth of $ X_t(\cdot,s)$ near $x$ means that the particles that at time $s$ were close together in the vicinity of $x$, have been widely separated at time $t$. We shall see that if $f\in C^2$ then $x\mapsto X^c_t(x,s)$ is at worst $\hf$-Hölder continuous---this is exactly the regularity obtained in Example \ref{ex:burgers_rarefaction} for Burgers' equation (cf.~\eqref{eq:burgersHolderEstimate}), and is therefore sharp. For general $f\in C^1$ we have the following regularity result. (For ease of exposition we assume that $f'$ is uniformly continuous; of course, for a fixed $u$ and $c$ such an assumption is unproblematic since $f'$ is only evaluated on the compact set $[c\wedge \inf u_0,c\vee \sup u_0]$.)

\begin{proposition}[Spatial regularity]\label{prop: particleSeparation}
    Let $u$ be the entropy solution of \eqref{eq:cl} with $u_0\in L^\infty\cap \BV_{\loc}(\R)$ and with $f'$ admitting a global nondecreasing modulus of continuity $\omega$. Then the flow of \eqref{eq:flow} admits the spatial regularity
\begin{align}\label{eq: generalControlOnSeparationOfParticles}
    X_t^c(y,s) - X_t^c(x,s) \leq \inf_{k>0}\bigg(y-x+\frac{1}{k}\int_{x}^{y}|u(z,s)-c|d z + \omega(2k)(t-s)\bigg),
\end{align}
for all $x\leq y$, $0\leq s\leq t$ and $c\in\R$.
\end{proposition}
\begin{proof}
For $k>0$ define
\begin{equation}\label{eq: etaCKEntropy}
    \eta_{c,k}(v)\coloneqq 1\vee\frac{|v-c|}{k}, \qquad q_{c,k}(v)\coloneqq \frac{1}{k} \sgn\big(v-[v]_{c-k}^{c+k}\big)\big(f(v)-f\big([v]_{c-k}^{c+k}\big)\big),
\end{equation}
where we have introduced the truncation
\begin{equation}\label{eq: truncation}
    \begin{split}
            [v]_a^b\coloneqq (a\vee v)\wedge b,
    \end{split}
\end{equation}
with $a\leq b$ in this last definition. By differentiating (weakly) we find $\eta_{c,k}'=f'q_{c,k}'$ and so $(\eta_{c,k},q_{c,k})_{k>0}$ forms a family of entropy pairs. Moreover, these pairs satisfy
\begin{align}\label{eq: algebraicIdenityForEtaCKEntropies}
    \eta_{c,k}(v)A(v,c)-q_{c,k}(v)= A\big([v]_{c-k}^{c+k},c\big)
\end{align}
for all $v\in \R$ and with $A$ as in \eqref{eq:velocityfield}. This identity, which in some sense generalizes \eqref{eq: algebraicIdentitySatisfiedByTheKruzkovEntropies}, is verified by checking the three cases $v-c<-k$, $|v-c|\leq k$, and $v-c>k$.

Now, since $\eta_{c,k}\geq 1$ we have the trivial inequality
\begin{align}\label{eq: trivialInequalityFeaturingEtack}
     y_t-x_t\leq \int_{x_t}^{y_t}\eta_{c,k}(u(z,t)) \,dz,
\end{align}
where $y_t\coloneqq X^c_t(y,s)$ and $x_t\coloneqq X^c_t(x,s)$. Using that $1\vee a\leq 1+ a$ whenever $a\geq 0$, we can bound the right-hand side of \eqref{eq: trivialInequalityFeaturingEtack} at time $t=s$ by
\begin{align*}
    \int_{x}^{y}\eta_{c,k}(u(z,s)) \,dz\leq y-x + \frac{1}{k}\int_{x}^{y}|u(z,s)-c| \,dz.
\end{align*}
And if we combine Lemma \ref{lemma: explicitLocalEntropyControl} with the identity \eqref{eq: algebraicIdenityForEtaCKEntropies} the growth-rate of this integral is for a.e.~$t> s$ bounded  by
\begin{align}\label{eq: boundOnLocalGrowthOfEtackEntropies}
    \frac{d}{ \,dt}\int_{x_t}^{y_t}\eta_{c,k}(u(z,t)) \,dz\leq&\,  A\big([u(y_t-,t)]_{c-k}^{c+k},c\big)-A\big([u(x_t+,t)]_{c-k}^{c+k},c\big).
    \end{align}
As $A(v,c)=\int_0^1f'(rv + (1-r)c)d r$ and $\big|[v]_{c-k}^{c+k}-[w]_{c-k}^{c+k}\big|\leq 2k$ for all $v,w\in\R$, the right-hand side of \eqref{eq: boundOnLocalGrowthOfEtackEntropies} is further bounded by
\begin{align}\label{eq: boundOnRightHandSideOfLocalEntropyGrowthOfEtack}
    \int_0^1\omega(2kr)\,dr\leq \omega(2k).
\end{align}
Since $t\mapsto \int_{x_t}^{y_t}\eta_{c,k}(u(z,t)) \,dz$ is weakly differentiable (by Lipschitz continuity) we conclude from the above analysis that it is bounded by the square bracket in \eqref{eq: generalControlOnSeparationOfParticles}. Then \eqref{eq: trivialInequalityFeaturingEtack} and the generality of $k>0$ gives the result.
\end{proof}

\begin{remark}\label{rem: extendingRegularityEstimatesToSAndC}
    Although we will not write it out explicitly, regularity estimates of the flow with respect to $s$ and $c$ can be deduced by similar means: For $t\geq s_2\geq s_1\geq 0$, the semi-group property implies that $X^c_t(x,s_1) = X^c_t(y,s_2)$ where $y \coloneqq X^c_{s_2}(x,s_1)$. Combined with $|x-y|\leq L(s_2-s_1)$, where $L$ is a local Lipschitz constant of $f$, the regularity in $s$ is controlled by that in $x$. Regularity estimates in $c$ are more complicated: The proof of Proposition \ref{prop: particleSeparation} can be repeated, but with $x_t$ and $y_t$ flowing according to the velocity fields $a_c$ and $a_{\tilde{c}}$ respectively. An appropriate modification of Lemma \ref{lemma: explicitLocalEntropyControl} is required which replaces the use of $\dot{y}_t\leq a_c(y_t-,t)$ with $\dot{y}_t\leq a_{\tilde{c}}(y_t-,t) \leq  a_c(y_t-,t)+\omega(|c-\tilde{c}|)$ where $\omega$ is a modulus of continuity of $f'$. In particular, the computations give similar regularity of $X^c_t(x,s)$ in $x$,~$s$~and~$c$.
\end{remark}

For certain moduli of continuity $\omega$, such as $\omega(h)=Ch^{\alpha}$ with $\alpha\in(0,1]$, the estimate \eqref{eq: generalControlOnSeparationOfParticles} can easily be explicitly minimized in $k>0$. And as we assume bounded $u_0$, we have $\int_{x}^{y}|u(z,s)-c|\,dz\leq (y-x)\|u(\cdot,s)-c\|_{L^\infty}\leq (y-x)\|u_0-c\|_{L^\infty}$ which can be used to further simplify \eqref{eq: generalControlOnSeparationOfParticles}. As an example, we present the following proposition where we consider linear $\omega$; this is the precise statement of the informal Theorem \ref{thm:flowRegular}.

\begin{proposition}[$\hf$-Hölder continuity of the flow]\label{prop: holderHalfRegularityWithC2Flux}
    Let $u$ be the entropy solution of \eqref{eq:cl} with $u_0\in L^\infty\cap \BV_{\loc}(\R)$ and where $f'$ is Lipschitz continuous with constant $L$. Then the flow of \eqref{eq:flow} admits the spatial regularity
\begin{align*}
    \bigl|X^c_t(y,s)-X^c_t(x,s)\bigr|\leq |y-x| + 2\sqrt{L(t-s)\norm{u_0-c}_{L^\infty}|y-x|}
\end{align*}
for all {$x,y\in\R$}, $0\leq s\leq t$ and $c\in\R$.
\end{proposition}
\begin{proof}
Assume without loss of generality that $x< y$. By arguing as in the proof of \eqref{prop: particleSeparation} we get
\begin{align}\label{eq: explicitControlOnSeparationOfParticlesInLipschitzCase}
    X^c_{t}(y,s) - X^c_{t}(x,s) \leq \inf_{k>0}\bigg[(y-x)\bigg(1+\frac{\|u_0-c\|_{L^\infty}}{k}\bigg)+ L(t-s)k\bigg],
\end{align}
where we used the bound from the discussion preceding the proposition and exploited the finer inequality $\int_0^1\omega(2kr)\,dr= 2Lk\int_0^1r\,dr=Lk$ instead of \eqref{eq: boundOnRightHandSideOfLocalEntropyGrowthOfEtack}. Minimizing the right-hand side over $k$ gives the minimizer
\begin{align*}
    k = \sqrt{\frac{(y-x)\norm{u_0-c}_{L^\infty}}{L(t-s)}}.
\end{align*}
Inserting this in \eqref{eq: explicitControlOnSeparationOfParticlesInLipschitzCase} gives the desired conclusion.
\end{proof}

\begin{remark}
    As seen in Example \ref{ex:burgers_rarefaction}, one can not generally expect the flow to be more than $\hf$-Hölder in space. While we won't explore this topic in detail, we make a few observations hinting at an underlying cause for this regularity loss: If $c$ does not intersect the range of $u$, then the flow $X^c$ is Lipschitz continuous in $x$; this follows by the proof of Proposition \ref{prop: particleSeparation} when exploiting that the right-hand side of \eqref{eq: boundOnLocalGrowthOfEtackEntropies} vanishes if $k<|u-c|$. Thus, roughness in $x\mapsto X^c(x,t)$ emerges only when $c$ intersects the range of $u$. It is \emph{not} shocks that cause this irregularity---particles merge at shocks, yielding constant regions in $X^c$. Rather, and as seen in Example \ref{ex:burgers_rarefaction}, rarefaction waves will pull particles apart and yield low regularity of the flow map. Note that this effect is also seen in characteristics, but unlike characteristics, the well-posedness of particle paths is preserved after this interaction.
\end{remark}


\section{A selection principle for conservation laws}\label{sec:selectionPrinciple}

In this section we show that well-posedness of the Filippov flow \eqref{eq:flow} for a given $c$ implies that $u$ satisfies the entropy inequality \eqref{eq:entropycondition} with Kruzkhov's entropy pair \eqref{eq:kruzkoventropy} for $k=c$. Consequently, if the flow is well-posed for all $c\in \R$ (or more generally, for a dense set in the range of $u$), then $u$ is the unique entropy solution of the conservation law \eqref{eq:cl}. This can be viewed as a novel selection criterion for solutions of scalar conservation laws.

\subsection{A representation formula for the continuity equation}

A key component of our approach is a representation formula for weak $L^\infty$-solutions of the continuity equation \eqref{eq:conteq} under the assumption that the velocity field $a$ generates a unique Filippov flow. By a weak solution of \eqref{eq:conteq} we mean a function $v \in L^\infty(\R \times \R_+)$ which satisfies
\begin{equation} \label{eq:continuity_eq_weak}
    \int_{\R_+} \int_\R (\partial_t \varphi + a \partial_x \varphi) v\, dx \,dt + \int_\R \varphi(x, 0) v_0(x)\, dx = 0
\end{equation}
for all $\varphi \in C^\infty_c(\R \times \R_+)$.
For the following theorem, and the calculations to come, we recall the abbreviation $X_t(x)=X_t(x,0)$.

\begin{theorem}[Representation formula] \label{thm:representation}
    Let $a \in L^1_{\loc}(\R_+; L^\infty(\R))$ generate a unique flow $X$ and assume that $a(\cdot,t)$ has one-sided limits for a.e.~$t > 0$. If $v \in L^\infty(\R\times \R_+)$ is a weak solution of \eqref{eq:conteq} with initial data $v_0 \in L^\infty(\R)$, then
    \begin{equation} \label{eq:representation_formula}
        v(\cdot, t) = (X_t)_\# v_0.
    \end{equation}
    More generally, $v(\cdot,t)= X_t(\cdot,s)_\# v(\cdot,s)$ for any $t\geq s\geq 0$.
\end{theorem}

\begin{remark}
    If $v_0$ and $X$ are such that the pushforward measure $(X_t)_\# v_0$ is absolutely continuous with bounded density, then Theorem~\ref{thm:representation} also gives \emph{existence} of a solution in $L^\infty(\R \times \R_+)$.
\end{remark}

\begin{remark} \label{rem:cont_with_source}
    Adding a term $w \in L^\infty(\R_+ \times \R)$ on the right hand side of \eqref{eq:conteq}, the conclusion of Theorem~\ref{thm:representation} still holds with
    \begin{equation*}
        v(\cdot, t) = (X_t)_\# v_0 + \int_0^t (X_t(\cdot,s))_\# w_s\, ds.
    \end{equation*}
\end{remark}

The proof of the representation formula \eqref{eq:representation_formula} goes via a duality technique, which requires a smoothing argument. Let $a\in L^1_\loc(\R_+; L^\infty(\R))$ be a velocity field which generates a unique flow $X$. Let $a^\eps \coloneqq a * \omega^\varepsilon$ be the spatial regularization of $a$ for a family of standard mollifiers $(\omega^\eps)_{\varepsilon > 0}$ in $C_c^\infty(\R)$.
Then $a^\eps \in L^1_\loc(\R_+; C^2_b(\R))$ generates a unique flow $X^\eps_t(x, s)$ which is well-defined for all $(x,s,t)\in \R\times\R_+^2$ (including the backwards-regime $0\leq t<s$). Moreover, the flow is continuous in $(x,s,t)$ and admits weak derivatives $\partial_x X^\eps,\partial_s X^\eps,\partial_t X^\eps\in L^1_{\loc}(\R\times\R_+^2)$; in fact, we have the regularity $\partial_x X^\eps\in C(\R\times\R_+^2)$ and the identity
\begin{equation*}
    \partial_s X^\varepsilon_t(x, s) = - a^\varepsilon(x, s) \partial_x X^\varepsilon_t(x, s)
\end{equation*}
for a.e.~$(x,s,t)\in \R\times\R_+^2$. We skip the proof of these facts as it consists of mostly standard calculations (see e.g.~\cite[Chapter 8]{ambrosio_gigli_savare_2005}). Let us now define the Borel measures on $\R$ characterized by
\begin{equation} \label{eq:sigma_def}
    \sigma_{\tau, t}^\varepsilon(I) \coloneqq X_\tau^\varepsilon(r_2, t)-X_\tau^\varepsilon(r_1, t) \quad \text{and}\quad  \sigma_{\tau, t}^0(I) \coloneqq X_\tau(r_2, t)-X_\tau(r_1, t)
\end{equation}
for any bounded interval $I\subset\R$ with endpoints $r_1<r_2$; observe that since the velocity field $a$ has a unique flow, Theorem \ref{thm: preliminariesOnFilippovFlow} implies that $x\mapsto X_\tau(x,t)$ is continuous and nondecreasing, so $\sigma^0_{\tau,t}$ is a well-defined nonnegative Radon measure without atoms (i.e., $\sigma^0_{\tau,t}(\{x\})=0$ for all $x\in\R$).


\begin{lemma} \label{lemma:conv_of_measure}
    Let $a\in L^1_\loc(\R_+; L^\infty(\R))$ generate a unique Filippov flow $X$. Then 
\begin{equation}\label{eq:strong_sigma_convergence}
        \lim_{\varepsilon \downarrow 0} \sigma_{\tau, t}^\varepsilon(I) = \sigma_{\tau, t}^0(I)
    \end{equation}
    for any $\tau\geq t\geq 0$ and any bounded interval $I \subset \R$.
\end{lemma}

\begin{proof}
    For fixed $x\in\R$, $t\geq 0$ the family of functions $(\tau \mapsto X_\tau^\varepsilon(x, t))_{\varepsilon > 0}$ is equicontinuous and uniformly bounded on any compact subset of $[t,\infty)$, so by Arzela--Ascoli and Filippov's stability theorem \cite[Theorem 3]{Filippov60} it converges (in $C_\loc$ sense) along a subsequence to a Filippov solution of $\dot{x}_\tau=a(x_\tau,\tau)$. Since we have assumed that $a$ generates a unique flow, the limit is unique and so the original family converges as $\varepsilon\downarrow0$. Thus, we get pointwise convergence $\lim_{\varepsilon\downarrow0} X_\tau^\varepsilon(x, t)=X_\tau(x,t)$ for all $\tau\geq t\geq 0$ and $x\in \R$.
    If $I$ is an interval with endpoints $r_1<r_2$, then
    \begin{align*}
        \lim_{\varepsilon\downarrow 0}\sigma_{\tau, t}^\varepsilon(I) = \lim_{\varepsilon\downarrow0} \Big(X_\tau^\varepsilon(r_2, t) - X_\tau^\varepsilon(r_1, t)\Big)=X_\tau(r_2, t) - X_\tau(r_1, t)=\sigma_{\tau, t}^0(I).
    \end{align*}
\end{proof}

\begin{proof}(Proof of Theorem \ref{thm:representation})
    Note first that the last line of the theorem is an immediate consequence of \eqref{eq:representation_formula} and the flow property $X_t(x,0) = X_t(X_s(x,0),s)$. Thus, we restrict our attention to $s=0$.

    Let $a^\eps$, $X^\eps$ and $\sigma^\eps$ be as defined before Lemma \ref{lemma:conv_of_measure}. For $\psi \in C^1_c(\R\times\R_+)$ we define
    \begin{equation} \label{eq:transport_solution}
        \varphi^\varepsilon(x, t) \coloneqq \int_t^\infty \psi(X^\varepsilon_\tau(x, t), \tau)\, d\tau.
    \end{equation}
    Let $R,T>0$ be such that $\supp\psi\subset[-R,R]\times[0,T]$. By enlarging $R$ if necessary, we also have $\supp\phi^\eps\subset [-R,R]\times[0,T]$ for all $\eps>0$. Moreover, $\varphi^\varepsilon$ satisfies
    \begin{equation} \label{eq:composition}
        \varphi^\varepsilon(X_t^\varepsilon(x), t) = \int_t^\infty \psi(X^\varepsilon_\tau(x), \tau)\, d\tau
    \end{equation}
    due to the flow property $X_\tau^\varepsilon(X_t^\varepsilon(x,0), t) = X_\tau^\varepsilon(x,0)$. The regularity of the flow $X^\varepsilon_\tau(x, t)$ (see the discussion following Remark \ref{rem:cont_with_source}) implies that $\varphi^\varepsilon$ has a derivative $\partial_x \varphi^\varepsilon \in C(\R \times \R_+)$ and a weak derivative $\partial_t \varphi^\varepsilon \in L_\loc^1(\R \times \R_+)$ with $\partial_t \varphi^\varepsilon(\cdot, t) \in C(\R)$ for a.e.~$t \in \R_+$. Differentiating both sides of \eqref{eq:composition} in $t$ yields
    \begin{equation*}
        (\partial_t \varphi^\varepsilon + a^\eps \partial_x \varphi^\varepsilon)(X^\varepsilon_t(x), t) = -\psi(X^\varepsilon_t(x), t)
    \end{equation*}
   for a.e.~$t \in \R_+$ and all $x \in \R$ (by continuity in $x$). Hence, since the map $x\mapsto X_t^\eps(x)$ is a diffeomorphism for all $t\geq0$, we conclude that
    \begin{equation} \label{eq:transport_equation}
        \partial_t \varphi^\varepsilon + a^\eps \partial_x \varphi^\varepsilon = - \psi
    \end{equation}
    a.e.~on $\R \times \R_+$. As $\varphi^\varepsilon$ can be used as a test function in \eqref{eq:continuity_eq_weak}, we have
    \begin{equation*} 
        \int_0^T \int_\R \psi v\, dx \,dt = \int_0^T \int_\R \underbrace{(a  - a^\varepsilon)\partial_x \varphi^\varepsilon v}_{\eqqcolon\,\mathcal{E}^\eps(x,t)}\, dx \,dt + \int_\R \varphi^\varepsilon(x, 0) v_0(x)\, dx.
    \end{equation*}
    We claim that $\lim_{\eps\downarrow 0}\|\mathcal{E}^\eps\|_{L^1(\R\times[0,T])}=0$, which would give
    \begin{equation*}
    \begin{split}
        \int_0^T \! \int_\R \psi v\, dx \,dt &= \lim_{\eps\downarrow0} \int_\R \varphi^\varepsilon(x, 0) v_0(x)\, dx
        =\lim_{\varepsilon\downarrow 0} \int_0^T \! \int_\R \psi(X^\varepsilon_t(x), t) v_0(x)\, dx \,dt \\
        &= \int_0^T\! \int_\R \psi(X_t(x), t) v_0(x)\, dx \,dt,
    \end{split}
    \end{equation*}
    and the formula \eqref{eq:representation_formula} would follow by approximation of $\psi\in C_c$ by $\psi \in C^1_c$.

    To prove the claim, note first that for any $0\leq t\leq \tau\leq T$ and $\eps>0$ we have
\begin{align}\label{eq: uniformBoundOnSigma}
    \sigma_{\tau,t}^{\eps}([-R,R]) = X_{\tau}^\eps(R,t)-X_{\tau}^\eps(-R,t) \leq 2(R + \|a\|_{L^1([0,T];L^\infty(\R))})\eqqcolon C_0,
\end{align}
which further gives
\begin{align*}
    \int_\R |\partial_x \varphi^\varepsilon(x,t) |\,dx\leq&\, \int_t^T\int_{-R}^{R}  \bigl|(\partial_x\psi)(X_\tau(x,t),\tau)\bigr|\,d\sigma_{\tau,t}^{\eps}(x)\,d\tau
    \leq  2T\|\partial_x\psi\|_{L^\infty}C_0.
\end{align*}
Thus, setting $C_1\coloneqq 2T\|\partial_x\psi\|_{L^\infty}C_0 \|v\|_{L^\infty}$ we get the bound
\begin{equation}\label{eq:commutator_x}
    \begin{split}
    \int_\R|\mathcal{E}^\eps(x,t)|\,dx\leq C_1\|a(\cdot, t)-a^\eps(\cdot, t)\|_{L^\infty} \leq 2C_1\|a(\cdot,t)\|_{L^\infty},
    \end{split}
\end{equation}
for a.e.~$t\in[0,T]$ and all $\eps>0$. By dominated convergence, the claim then follows if we can show that $\lim_{\eps\downarrow 0}\int_{\R}|\mathcal{E}^\eps(x,t)|\,dx=0$ for a.e.~$t\in[0,T]$.

 Fix any $t \in [0,T]$ for which $a(\cdot, t)$ has one-sided limits and is bounded. For $\kappa>0$ let $J^{\kappa}$ be the closed (and countable) set of points where $a(\cdot, t)$ has a jump discontinuity of size at least $\kappa$, that is,
    \begin{equation*}
        J^{\kappa} = \bigl\{x \in \R : \abs{a(x-, t) - a(x+, t)} \geq \kappa\bigr\}.
    \end{equation*}
    For $\delta > 0$ define the open set $J^\kappa_\delta=\{x\in\R: \mathrm{dist}(x,J^\kappa)<\delta\}$ and note that
    \begin{equation*}
        \limsup_{\varepsilon \downarrow 0}\|a(\cdot, t) - a^\varepsilon(\cdot, t)\|_{L^\infty(K \setminus J^{\kappa}_\delta)}  \leq \kappa
    \end{equation*}
    for any compact set $K\subset\R$. In particular, with $C_1$ as in \eqref{eq:commutator_x} we find that
    \begin{equation*}
        \begin{aligned}
            \limsup_{\varepsilon \downarrow 0} \int_{\R\setminus J^{\kappa}_\delta} |\mathcal{E}^\eps(x,t)| \, dx & \leq C_1  \limsup_{\varepsilon\downarrow 0}\|a(\cdot, t) - a^\varepsilon(\cdot, t)\|_{L^\infty([-R,R] \setminus J^{\kappa}_\delta)} \leq C_1 \kappa.
        \end{aligned}
    \end{equation*}
    On the other hand, setting $C_2(t)\coloneqq 2\|a(\cdot,t)\|_{L^\infty}\|v\|_{L^\infty}\|\partial_x\psi\|_{L^\infty}$ we have
    \begin{equation} \label{eq:jump_part}
        \begin{aligned}
        \int_{J^{\kappa}_\delta} |\mathcal{E}^\eps(x,t)| \, dx
            & \leq C_2(t)\int_t^T \sigma_{\tau, t}^\varepsilon(J^{\kappa}_\delta \cap [-R,R])\,d\tau.
        \end{aligned}
    \end{equation}
    Since $J^{\kappa}_\delta\cap[-R,R]$ is a finite union of bounded intervals, Lemma~\ref{lemma:conv_of_measure} ensures that
    \begin{equation*}
        \lim_{\varepsilon \downarrow 0} \sigma_{\tau, t}^\varepsilon\big(J^{\kappa}_\delta \cap [-R,R]\big) = \sigma_{\tau, t}^0\big(J^{\kappa}_\delta \cap [-R,R]\big)
    \end{equation*}
    for all $\tau \in [t, T]$. In light of the uniform bound \eqref{eq: uniformBoundOnSigma} and dominated convergence, these limits can be taken inside the integral in \eqref{eq:jump_part}, and so we conclude that
    \begin{equation*}
        \limsup_{\varepsilon \downarrow 0} \int_\R |\mathcal{E}^\eps(x,t)|\, dx \leq C_2(t)\int_t^T  \sigma_{\tau, t}^0(J^{\kappa}_\delta \cap [-R,R])\, d\tau + C_1 \kappa.
    \end{equation*}
    Since this was true for arbitrary $\delta > 0$ we may pass to the limit. Using that
    \begin{equation*}
        \lim_{\delta \downarrow 0} \sigma_{\tau, t}^0(J^{\kappa}_\delta \cap [-R,R]) = \sigma_{\tau, t}^0(J^{\kappa} \cap [-R,R])  = 0,
    \end{equation*}
    which follows as $J^{\kappa}$ is closed and countable while $\sigma_{\tau,t}^0$ has no atoms, we arrive at
    \begin{equation*}
        \limsup_{\varepsilon \downarrow 0} \int_\R |\mathcal{E}^\eps(x,t)|\, dx \leq C_1 \kappa,
    \end{equation*}
    after yet another use of dominated convergence. Finally, as $\kappa > 0$ was arbitrary we conclude that $\lim_{\eps\downarrow 0}\int_\R |\mathcal{E}^\eps(x,t)|\, dx =0$.
\end{proof}

\subsection{Well-posed flows implies entropy condition}
A consequence of Theorem~\ref{thm:representation} is that $u - c$ can be represented as a $X^c$-pushforward of previous values. This relation between the flow and the solution allows us to prove the following.

\begin{proposition} \label{prop:entropy_ineq}
    Let $f\in C^1(\R)$, let $u\in L^\infty(\R\times\R_+)$ be a weak solution of \eqref{eq:cl} with initial data $u_0\in L^\infty(\R)$, and assume that $u(t)\in \BV_{\loc}(\R)$ for a.e.~$t\geq 0$. Let $c\in\R$, and assume that $a_c\coloneqq A(u,c)$ has a unique Filippov flow $X^c$. Let $(\eta_c,q_c)$ be the Kruzkov entropy pair defined by \eqref{eq:kruzkoventropy}. Then $u$ satisfies the entropy inequality
    \begin{equation}\label{eq:kruzkov_entropy_ineq}
    \int_{\R_+}\int_\R \eta_c(u)\partial_t\phi + q_c(u)\partial_x\phi\,dx\,dt + \int_\R \eta_c(u_0(x))\phi(x,0)\,dx \geq 0
\end{equation}
    for all $0\leq \varphi\in C^\infty_c(\R\times\R_+)$.
\end{proposition}
\begin{proof}
    The function ${a_c(\cdot, t)\coloneqq A(u(\cdot,t),c)}$ has one-sided limits for a.e.~$t \in \R_+$ since $u \mapsto A(u, c)$ is continuous and $u(\cdot, t) \in \BV_\loc(\R)$. Therefore, by Theorem~\ref{thm:representation}, the function $u - c$ can be written as
    \begin{equation*}
        u(\cdot, t) - c = X_t^c(\cdot,s)_\# (u(\cdot,s) - c)
    \end{equation*}
    for any $t\geq s\geq 0$, where $X^c$ is the flow generated by $a_c$. For $\varepsilon>0$ and $0\leq \phi \in C^\infty_c(\R\times \R_+)$ we therefore get
    \begin{equation*}
        \begin{aligned}
           & \int_0^\infty\!\int_\R \phi(x, t) \eta_c\big(u(x, t)\big)\, dx\, dt
            =\sup_{\substack{\theta\in C_c(\R\times \R_+) \\ \abs{\theta} \leq 1}} \int_0^\infty\!\int_\R \phi(x, t) \theta(x,t) (u(x, t) - c)\, dx\, dt\\*
        & = \sup_{\substack{\theta\in C_c(\R\times \R_+) \\ \abs{\theta} \leq 1}} \Bigg[\begin{aligned}[t] \int_\varepsilon^{\infty}\!\int_\R \phi\big(X^c_{t}(x, t-\varepsilon), t\big) \theta\big(X^c_{t}(x, t-\varepsilon),t\big) (u(x, t-\varepsilon) - c)\, dx\, dt&\\*
         + \int_0^\varepsilon\! \int_{\R}\phi\big(X^c_{t}(x), t\big) \theta\big(X^c_{t}(x),t\big) (u_0(x) - c)\,dx\, dt&\Bigg]\end{aligned}\\*
        & \leq \int_\eps^\infty\!\int_\R \phi\big(X^c_{t}(x, t-\eps), t\big) \eta_c(u(x, t-\eps))\, dx\, dt + \int_0^\varepsilon\!\int_\R \phi\big(X^c_{t}(x), t\big) \eta_c\big(u_0(x)\big)\, dx\, dt \\*
        & = \int_0^\infty\!\int_\R \phi\big(X^c_{t+\varepsilon}(x, t), t+\varepsilon\big) \eta_c\big(u(x, t)\big)\, dx\, dt + \int_0^\varepsilon\!\int_\R \phi\big(X^c_{t}(x), t\big) \eta_c\big(u_0(x)\big)\, dx\, dt.
        \end{aligned}
    \end{equation*}
    Rewriting and dividing by $\varepsilon$ we obtain
    \begin{equation} \label{eq:entropy_difference_ineq}
        \begin{aligned}
            0 & \leq  \int_0^\infty \int_\R \bigg(\frac{\phi\bigl(X^c_{t+\varepsilon}(x, t), t+\varepsilon\bigr) - \phi(x, t)}{\varepsilon} \bigg)\eta_c\big(u(x, t)\big)\, dx \,dt \\
            & \quad + \frac{1}{\eps}\int_0^\varepsilon\int_\R \phi\big(X^c_{t}(x), t\big) \eta_c\big(u_0(x)\big)\, dx\, dt.
        \end{aligned}
    \end{equation}
    The bracket in the first integrand can be written as
    \begin{equation*}
        \int_0^1\partial_t\phi(x, t+r\varepsilon)+ \partial_x\phi\big(x+r(X_{t+\varepsilon}^c(x,t)-x), t+\varepsilon\big) \bigg(\frac{X^c_{t+\varepsilon}(x, t) - x}{\varepsilon}\bigg) \, dr,
    \end{equation*}
   which, as $\eps\downarrow 0$, tends to $\partial_t \phi + a_c\partial_x \phi$ in $L^1(\R\times\R_+)$ due to Lemma  \ref{lemma:convergence_universal_representative} and $\phi$ being $C^\infty_c(\R\times\R_+)$. Thus, letting $\eps\downarrow 0$ in \eqref{eq:entropy_difference_ineq} we conclude that
    \begin{equation*}
        \begin{aligned}
            0 &\leq \int_0^\infty\! \int_\R \Big(\partial_t \phi(x, t) + a_c(x, t)\partial_x \phi(x, t) \Big) \eta_c\big(u(x, t)\big)\, dx \,dt  + \int_\R \phi(x, 0) \eta_c\big(u_0(x)\big)\, dx,
        \end{aligned}
    \end{equation*}
    which, due to the identity \eqref{eq: algebraicIdentitySatisfiedByTheKruzkovEntropies}, is exactly \eqref{eq:kruzkov_entropy_ineq}.
\end{proof}

\begin{remark}
As mentioned in Remark \ref{rem:main_thm_generalization}, the above can be generalized to merely $f\in \mathrm{Lip}(\R)$ for those $c\in\R$ where $f'(c)$ exists.
\end{remark}

\begin{proof}(Proof of (ii) $\Rightarrow$ (i) in Theorem \ref{thm:main})
    If the Filippov flow $X^c$ for $a_c$ is unique for all $c \in \R$ it follows by Proposition \ref{prop:entropy_ineq} that the Kruzkhov entropy inequality holds for all $c$. Consequently, the weak solution $u$ is the entropy solution.
\end{proof}

We conclude this section with an example where we use uniqueness of Filippov flow as a selection principle.

\begin{example}
    Consider the problem
    \begin{equation} \label{eq:selcetion_prob}
        \partial_t u + \partial_x \biggl(\frac{1}{2}u^2\biggr) = 0, \qquad u_0(x) =
        \begin{cases}
            3 & \text{for } 0 < x < 1 \\
            1 & \text{else}
        \end{cases}
    \end{equation}
    for $t\in [0,1]$, and the two weak solutions
    \begin{equation*}
        u(x, t) =
        \begin{cases}
            \frac{x}{t} &  \text{if } t < x < 3t \\
            3 & \text{if } 3t < x < 1 + 2t \\
            1 & \text{else}
        \end{cases}
        \qquad \text{and} \qquad \tilde{u}(x, t) = u_0(x - 2t),
    \end{equation*}
    that is, the entropy solution and the traveling wave solution. For $c = 0$ one can check that both solutions have unique particle paths solving \eqref{eq:flow} with a right-hand side given by $\frac{1}{2}u$ and $\frac{1}{2}\tilde{u}$. However, with $c = 2$ there are infinitely many solutions for $\frac{1}{2}(\tilde{u} + 2)$ starting at $x = 0$: A particle may follow the nonentropic shock and, at any time $t \geq 0$, branch off to either side (rightmost picture in Fig.~\ref{fig:selection_non-entropic}). On the other hand, the right-hand side $\frac{1}{2}(u + 2)$ still generates a unique flow of particle paths (Fig.~\ref{fig:selection_entropic}). In other words, the paths are well-defined except when $c$ intersects the nonentropic shock.

    \begin{figure}
        \centering
        \begin{subfigure}[t]{\textwidth}
            \centering
            \includegraphics[width=0.32\textwidth]{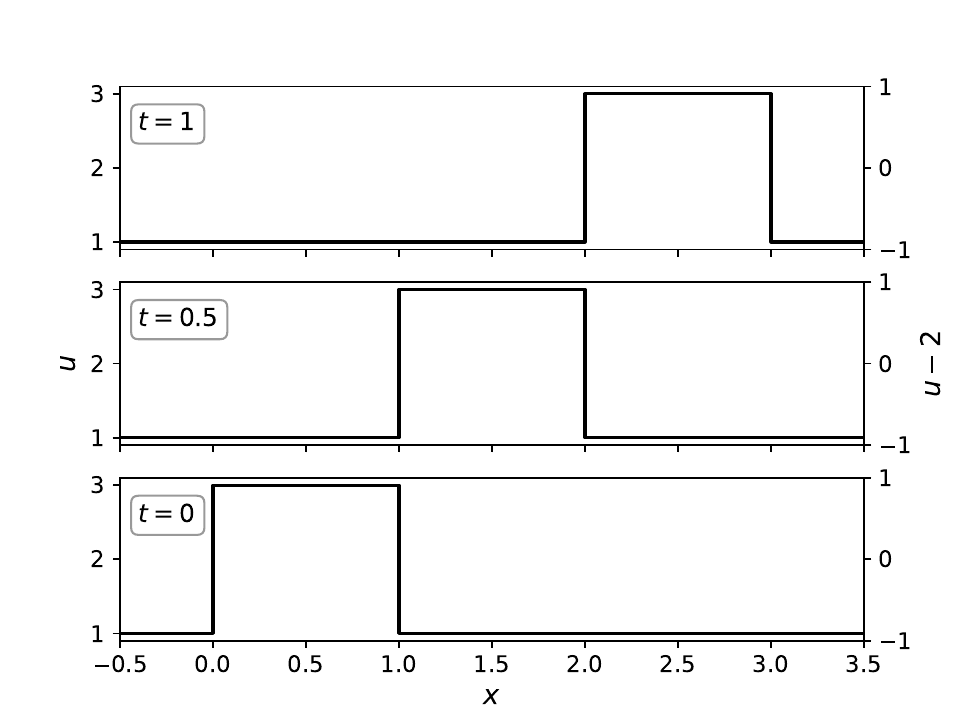}
            \includegraphics[width=0.32\textwidth]{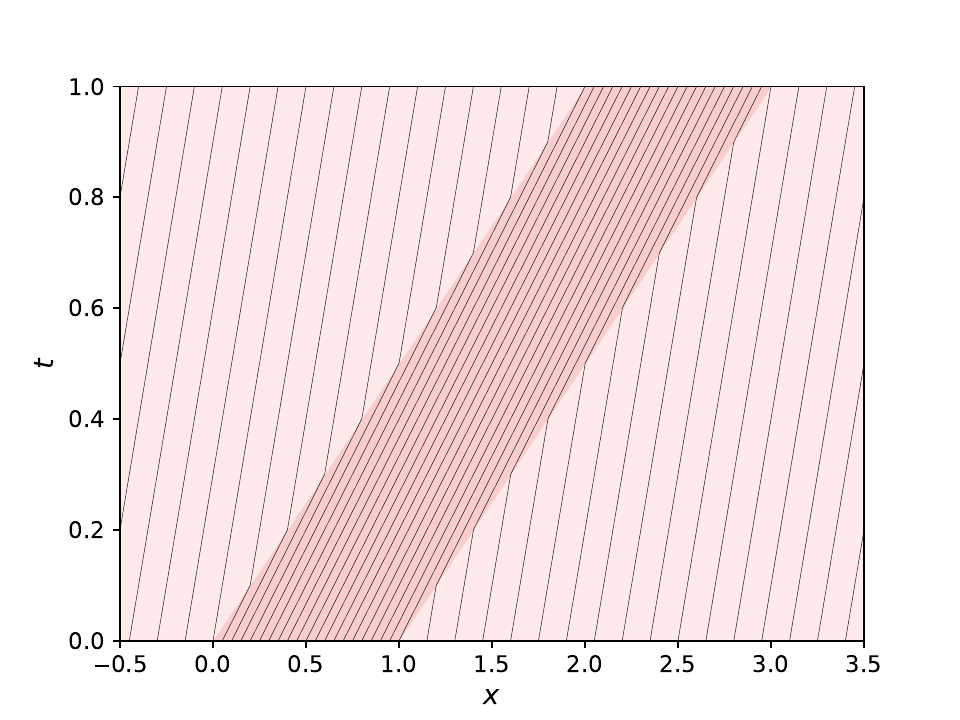}
            \includegraphics[width=0.32\textwidth]{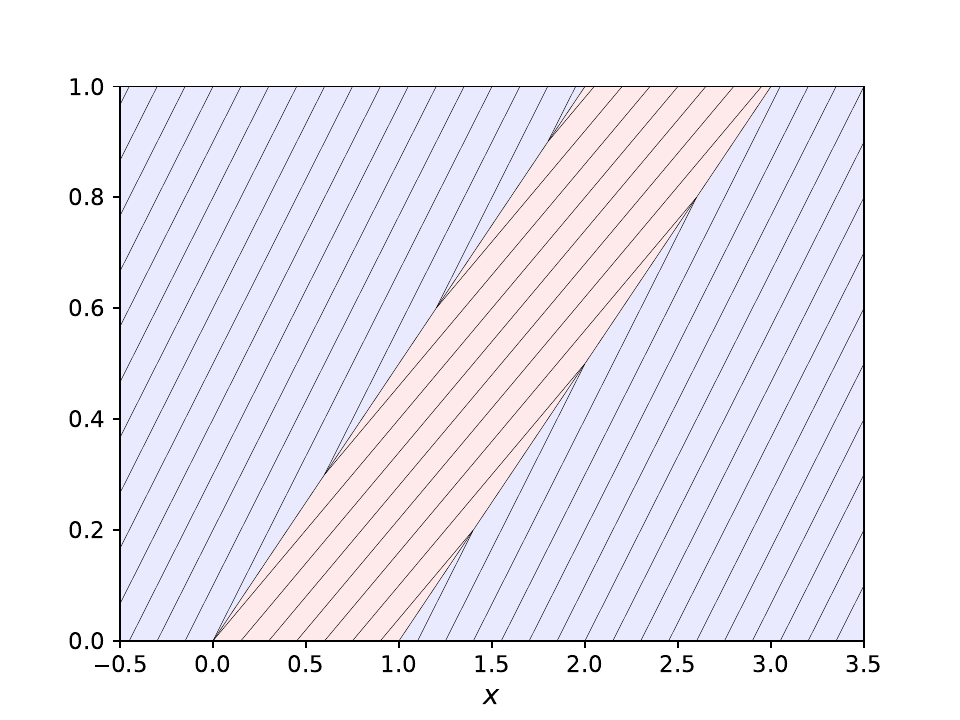}
            \caption{Nonentropic solution $\tilde{u}$.\label{fig:selection_non-entropic}}
        \end{subfigure}

        \begin{subfigure}[t]{\textwidth}
            \centering
            \includegraphics[width=0.32\textwidth]{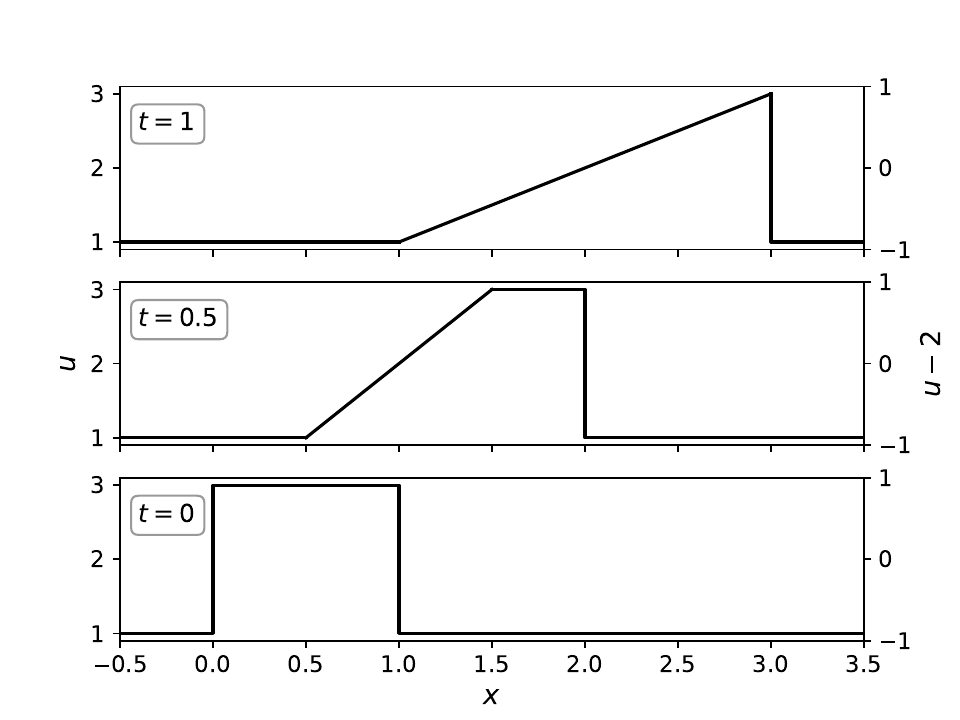}
            \includegraphics[width=0.32\textwidth]{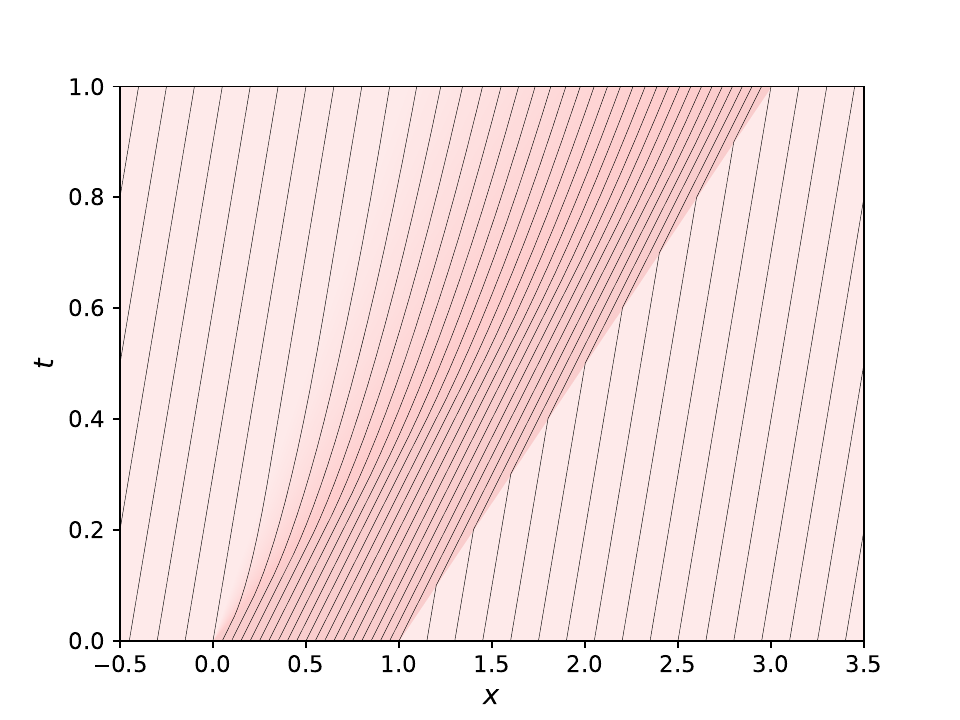}
            \includegraphics[width=0.32\textwidth]{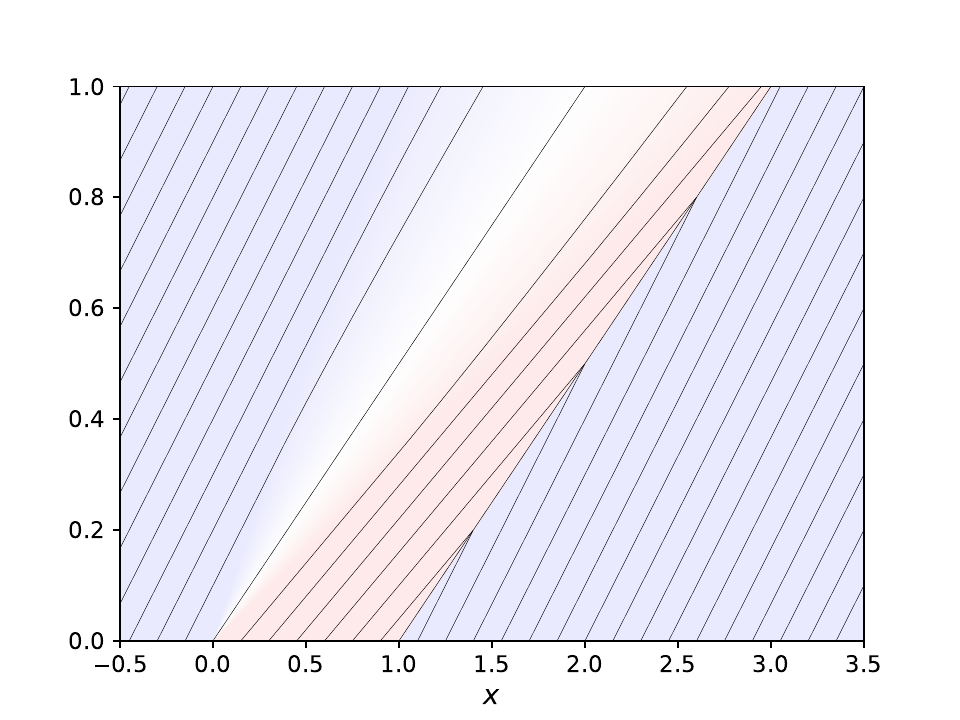}
            \caption{Entropy solution ${u}$.}\label{fig:selection_entropic}
        \end{subfigure}
        \caption{Nonentropic solution $\tilde u$ (top row) and entropy solution $u$ (bottom row). Solution at $t=0, 0.5 ,1$ (left), particle paths for $a_c$ with $c=0$ (middle), and with $c=2$ (right). The figures are colored according to the value of $u-c$, with red representing positive and blue negative values. The initial density of the particles has been chosen to be proportional to $|u_0-c|$.}
        \label{fig:selection}
    \end{figure}

\end{example}

\section{Conclusions and future work}
We have shown that there is a close connection between the entropy selection principle for nonlinear conservation laws \eqref{eq:cl}, and the well-posedness of associated linear continuity equations \eqref{eq:conteq}. The continuity equation arises as a linearization of the nonlinear conservation law; thus, one way of interpreting our result is that the entropy selection principle is equivalent to \emph{linear stability}, in the above sense. The proof of sufficiency, given in Section \ref{sec:uniquenessOfFlow}, goes via a novel and detailed analysis of particle paths (solutions of the associated ODE \eqref{eq:flow}). As a byproduct of this analysis, we derive a sharp regularity estimate on the flow of the ODE. The proof of necessity, given in Section \ref{sec:selectionPrinciple}, first establishes a representation formula for the continuity equation by a renormalization technique, and then uses this to estimate the entropy production of the weak solution.

As seen in Example \ref{ex:multiple_dimensions}, our result is not valid in multiple space dimensions. Although it should be possible---at least in principle---to generalize our proof of necessity to multiple dimensions, our proof of sufficiency is fundamentally one-dimensional. Any generalization to multiple dimensions would therefore require an entirely different approach. Generalizations to spatially dependent fluxes or nonlocal fluxes would be equally interesting.

Any new stability property of a PDE opens up the possibility of novel proofs of convergence of various types of approximations. We aim to utilize the ideas put forth in this paper to prove convergence of numerical and stochastic approximations of \eqref{eq:cl} in forthcoming works.

The stability of particle paths with respect to perturbations in the initial data $u_0$ and the flux $f$ is also a natural topic to explore. Results in this direction will be presented in an upcoming paper.

A corollary of our Main Theorem and the representation formula in Theorem~\ref{thm:representation} is that for any $v_0\in L^\infty(\R)$, there exists at most one bounded weak solution of the ``linearized'' continuity equation \eqref{eq:conteq} with velocity $a=a_c$. However, we have been unable to prove the converse---that uniqueness of the continuity equation implies that $u$ is the entropy solution of \eqref{eq:cl}.  Results in this direction would be interesting, since they would broaden our understanding of the connection between continuity equations \eqref{eq:conteq} and the associated ODE \eqref{eq:flow}. (See e.g.~Ref.~\cite{ambrosio_crippa_2008} for results in this direction, but with \emph{nonnegative} solutions.)

\section*{Acknowledgements}
U.~S.~Fjordholm and O.~H.~Mæhlen were partially supported by the Research Council of Norway project \textit{INICE}, project no.~301538.

\bibliography{bibliography}
\bibliographystyle{plain}
\end{document}